\documentclass[a4paper,reqno,11pt]{amsart}
\usepackage{amsmath}
\usepackage{amsthm}
\usepackage{amscd}
\usepackage{amssymb}
\usepackage{amsfonts}
\usepackage{latexsym}
\usepackage[mathscr]{eucal}

\newtheorem{thm}{Theorem}
\newtheorem{prop}[thm]{Propositoin}
\newtheorem{lemma}[thm]{Lemma}

\theoremstyle{definition}
\newtheorem{defn}[thm]{Definition}

\newtheorem{exam}{Example}[section]
\newtheorem{rem}{Remark}[section] 

\def\diag{\hbox{\rm diag}}

\def\O{{\mathcal O}}

\def\Aut{\hbox{\rm Aut}}
\def\K{{\mathcal K}}

\def\F{{\mathcal F}}
\def\L{{\mathcal L}}

\def\C{{\mathcal C}}

\def\comp{{\mathbb {C}}}

\def\cst{{C${}^*$}}

\begin{document}
\title{Dimension groups for self-similar maps and 
matrix representations of the cores of the associated 
$C^*$-algebras}

\author{Tsuyoshi Kajiwara}
\address[Tsuyoshi Kajiwara]{Department of Environmental and
Mathematical Sciences,
Okayama University, Tsushima, Okayama, 700-8530,  Japan}

\author{Yasuo Watatani}
\address[Yasuo Watatani]{Department of Mathematical Sciences,
Kyushu University, Motooka, Fukuoka, 819-0395, Japan}

\maketitle
\hyphenation{cor-re-spond-ences}
\begin{abstract}We introduce a dimension group for a self-similar map as 
the ${\rm K}_0$-group of the core of the \cst -algebra associated 
with the self-similar map together with the canonical endomorphism. 
The key step for the computation is an explicit description of the 
core as the inductive limit using 
their matrix representations over the coefficient algebra, 
which can be described explicitly by the singularity structure 
of branched points. We compute that the dimension group 
for the tent map is isomorphic to the countably generated 
free abelian group ${\mathbb Z}^{\infty}\cong {\mathbb Z}[t]$ 
together with the unilatral shift, i.e. the 
multiplication map by $t$ as an abstract group. 
Thus the canonical endomorphisms on the ${\rm K}_0$-groups are not 
automorphisms in geneal. This is a different point compared with 
dimension groups for topological Markov shifts. We can count the 
singularity structure in the dimension groups. 
\end{abstract}

\section
{Introduction}

Dimension groups for topological Markov shifts were introduced and studied by 
Krieger in \cite{Kr1} and \cite{Kr2} motivated by the K-groups 
for the AF-subalgebras of the associated Cuntz-Krieger algebras \cite{CK}.
The dimension group is an ordered abelian group with a canonical automorphism.
The dimension group for topological Markov shifts is very important, because 
it is a complete invariant up to shift equivalence. 
Matsumoto introduced a class of 
\cst-algebras associated with general subshifts in \cite{Ma1} and studied 
dimension groups for subshifts in \cite{Ma2}. 
In our paper we shall introduce dimension groups for self-similar maps as 
the $K_0$-groups of the cores of 
\cst -algebras associated with self-similar maps,  
where the core is the fixed point subalgebra under the gauge action 
of the \cst-algebra associated with a self-similar map. 
But canonical automorphisms on dimension goups can be generalized as only 
endomorphisms in the case of self-similar maps. The key step for 
the computation is an explicit description of the cores as the inductive limit 
using matrix representations over the coefficient algebra, 
which heavily depends on the structure of branched points.

Self-similar maps on compact metric spaces produce many interesting 
self-similar sets like Sierpinski Gasket, see for example \cite{F},
\cite{Kig}.  We constructed \cst-algebras for self-similar maps using 
\cst-correspondence and Pimsner construction (\cite{KW1}). 
\cst-algebras associated with self-similar maps are 
simple, purely infinite and in UCT class, and are classifiable by 
their K-groups (\cite{KW1}). The algebra of the continuous functions 
on the self-similar set is shown to be 
maximal abelian in \cite{KW6}.

In \cite{KW3}, \cite{KW5}, we classified traces and ideals of the core of
the \cst-algebras associated with self-similar maps under some
conditions, i.e. Assumption A and Assumption B, 
concerning branched point sets. If 
there exists no branched points,then 
the core of the \cst-algebra associated with a self-similar map is simple
and has a unique tracial state.  
On the other hand, if there exists a branched point, then the core has 
many finite discrete traces. 
Moreover, in \cite{KW5}, we constructed the matrix representation of the
core over the coefficient algebra, which makes the classification of 
ideals of the core possible.

In the paper, we present an explicit form of matrix
representations over the coefficient algebra 
of the finite cores of the \cst-algebra associated with
self-similar maps under the Assumption B. Here  the Assumption B 
roughly means that the set of branched points is a finite set and 
every orbit has no branched point or one branched point at most. Many 
typical examples like a tent map satisfies the Assumption B. 
We express the 
${\rm K}$ groups of the cores of
\cst-algebras associated with tent map and Sierpinski Gasket
respectively as inductive limit.  
We compute that the ${\rm K}_0$-group of the core of the 
\cst -algebras associated with the tent map is isomorphic to 
the countably generated 
free abelian group ${\mathbb Z}^{\infty}$ as an abstract group. 
Moreover, we can prove that the canonical endomorphism of  
the dimension group  for tent map case is the unilatral shift on 
${\mathbb Z}^{\infty}$, which is not surjective. On the 
otherhand, if a self-similar map has no branched points, then 
the canonical endomorphism of  the dimension group 
becomes an automorphism.

We refer other interesting approach on dimension groups for interval maps 
by F. Shultz \cite{Sh1},\cite{Sh2}  and Deaconu-Shultz \cite{DS}.

The contents of the paper is as follows: Section 1 is an introduction. 
In Section 2, we recall  
preliminary results on self-similar maps and \cst-correspondences.  
In Section 3, we summarize the construction of the matrix representation
over the coefficient algebra of the core following \cite{KW5}.  
In Section 4, we present explicit
description of the matrix representations of the finite cores for the
case of tent map and Sierpinski Gasket.  In Section 5, we present
explicit calculation of ${\rm K}$-groups of the cores for the case 
of tent map and Sierpinski Gasket using explicit description of the
finite core.  In Section 6, we define the  cnanonical endomorphism on the 
dimension group in general. If the the Hilbert module is finitely generated, 
then the canonical endomorphism on the 
dimension group is induced by the dual action of the gauge action. 
In the case of the tent map, we show that the canonical endomorphism is not 
surjective by the explicit calculation.  

This work was supported by JSPS KAKENHI Grant Number JP16K05178, 
JP23540242, and JP25287019.

\section
{Self-similar maps and their \cst -correspondences}
\label{sec:Self}
Let $(\Omega,d)$ be a (separable) complete  metric space.  A map 
$f:\Omega \rightarrow \Omega$ is said to be 
a proper contraction if there exists a constant $c$ and $c'$ with
 $0<c' \leq c <1$ such that
$0<  c'd(x,y) \leq d(f(x),f(y)) \le cd(x,y)$ for any  $x$, $y \in \Omega$.

We consider a family $\gamma = (\gamma_0,\dots,\gamma_{N-1})$ of $N$ proper
contractions on $\Omega$. We always assume that $N \geq 2$. Then there
exists a unique non-empty compact set $K \subset \Omega$ which is
self-similar in the sense that $K = \bigcup_{i=0}^{N-1} \gamma_i(K)$. See
Falconer \cite{F} and Kigami \cite {Kig} for more on fractal sets. 

In this paper, we usually forget an ambient space $\Omega$ as in \cite{KW1} 
and start with the following: Let $(K,d)$ be a compact metric set and  
 $\gamma = (\gamma_{0},\dots,\gamma_{N-1})$ be a family of $N$ proper contractions 
on $K$. We say that 
$\gamma$ is a self-similar map on $K$ if
$K = \bigcup_{i=0}^{N-1} \gamma_i(K)$.  Throughout the paper we assume that 
$\gamma$ is a self-similar map on $K$. 

\begin{defn}  \rm 
We say that $\gamma$ satisfies the open set condition if there exists an
open subset $V$ of $K$ such that $\gamma_j(V) \cap \gamma_k(V)=\phi$ for
$j \ne k$ and  $\bigcup_{i=0}^{N-1} \gamma_i(V) \subset V$.  See a book
 \cite{F} by Falconer. . 
\end{defn}

Let $\Sigma  = \{\,0,\dots,N-1\,\}$.  For $k \ge 1$, we put $\Sigma^k =
\{\,0,\dots,N-1\,\}^k$. For a self-similar map $\gamma$ on a compact 
metric space $K$, we
introduce the following subsets of $K$:
\begin{align*}
B_{\gamma} = &\{\,b \in K\,|\, b =\gamma_{i}(a)=\gamma_{j}(a), \,\,
\text{for some}\,\, a \in K\, \text{and } \, i \ne j\,\, \}, \\
C_{\gamma} = &\{\,a\in K\,|\,\gamma_{i}(a)=\gamma_{j}(a), \,\,
\text{for some}\,\, a \in K\, \text{and } \, i \ne j\,\, \} 
= \{\,a\in K\,|\,\gamma_j(a) \in B_{\gamma}\,\,\text{for
some}\,\,j\,\},  \\
P_{\gamma} = & \{\,a\in K\,|\, \exists k \ge 1,\,\exists (j_1,\dots,j_k)
\in \Sigma^k\,\,
  \text{such that}\,\, \gamma_{j_1}\circ \cdots \circ \gamma_{j_k}(a) \in
B_{\gamma}\}.  
\end{align*}

Any element $b$ in $B_{\gamma}$ is called a branched point and 
any element $c$ in $C_{\gamma}$ is called a branched value. 
We call $B_{\gamma}$ the branch set of $\gamma$, $C_{\gamma}$ the branch
value set of $\gamma$, and $P_{\gamma}$ the postcritical set of
$\gamma$. 

Throughout the paper, we assume that a self-similar map $\gamma$ on $K$
satisfies the following Assumption B: 
\begin{enumerate}
  \item There exists a continuous map $h$ of $K$ to $K$ which
        satisfies $h(\gamma_j(y))=y$ $(y\in K)$ for each $j$.
  \item The set $B_{\gamma}$ is a finite set.
  \item $B_{\gamma} \cap P_{\gamma} = \emptyset$. 
  
\end{enumerate}
The conditon (3) means that every orbit has no branched point or one branched point at most.

If (2) is replaced by a stronger condition 

(2)' The sets $B_{\gamma}$ and $P_{\gamma}$ are  finite sets, 

then it is exactly the assumption A studied in \cite{KW3}. 

Therefore $B_{\gamma}$ is the set of $b \in K$ such that $h$ is not 
locally homeomorphism at $b$, that is, $B_{\gamma}$ is the set of 
the branched points of $h$ in the usual sense. Moreover 
$C_{\gamma}= h(C_{\gamma})$ is the branch value set of $h$ and 
$\cap P_{\gamma}$ is the postcritical set of $h$ in the usual sense.

Many important examples including tent map and Sierpinski gasket 
satisfy the assumption B above.
If we assume that $\gamma$ satisfies the assumption B, then 
$K$ does not have any isolated points and $K$ is not countable.

{\bf Notation.}  
For fixed $b \in B_{\gamma}$, we denote by $e_{b}$ the number of $j$
such that $b =\gamma_{j}(h(b))$.  We call $e_{b}$ the branch index at
$b$. Then $e_{b}$ the branch index at $b$ for $h$ in the usual sense.  
We label these indices $j$ so that 
$$
\{ j \in \Sigma\ |\ b =\gamma_{j}(h(b))\} = 
 \{j(b,0),j(b,1),\dots, j(b,{e_b-1})\}
$$
satisfying $j(b,0)< j(b,1)<
\cdots <j(b,e_{b}-1)$. 

\begin{exam}\cite{KW1} 
 \label{ex:tent} {\rm (tent map) } 
Define a family $\gamma = (\gamma_0,\gamma_1)$ of proper contractions
on $K=[0,1]$ by $\gamma_0(y) = (1/2)y$ and $\gamma_1(y)= 1-(1/2)y$. Then 
$\gamma$ is a self-similar map. 
We see  that $B_{\gamma}=\{\,1/2\,\}$, $C_{\gamma} = \{\,1\,\}$. 
We define the ordinary tent map $h$ on $[0,1]$ by 
\[  h(x) = \begin{cases}
          2x & \quad  0 \le x \le 1/2  \\
          -2x + 2 & \quad 1/2 \le x \le 1
         \end{cases}
\]
Then $\gamma$ satisfies Assumption B. 
$\gamma = (\gamma_0,\gamma_1)$ is the inverse branches of the tent map $h$.
We see that $B_{\gamma}$ and $C_{\gamma}$ are finite sets.   
We note  that $h(1/2) = 1$, $h(1) = 0$, $h(0) =0$. Hence  $P_{\gamma} =
 \{\,0,1\,\}$ and $B_{\gamma} \cap P_{\gamma} = \emptyset$.    
\end{exam}

\begin{exam}\cite{KW1} {\rm (Sierpinski gasket)} Consider a regular 
triangle ${\mathcal T}$ on ${\bf R}^2$ with three
vertices ${\rm P} = (1/2, \sqrt{\,3}/2)$, ${\rm Q} = (0,0)$ and 
${\rm R} =(1,0)$.  
Let $S=(1/4,\sqrt{\,3}/4)$, ${\rm T}=(1/2,0)$ and ${\rm U}=(3/4,\sqrt{\,3}/4)$.
Define proper contractions $\tilde{\gamma}_0$, $\tilde{\gamma}_1$ and $\tilde{\gamma}_2$ on the triangle ${\mathcal T}$ by 
\[
\tilde{\gamma}_0(x,y) =
\left(\frac{x}{2}+\frac{1}{4},\frac{1}{2}y\right), \quad
\tilde{\gamma}_1(x,y)=\left(\frac{x}{2},\frac{y}{2}\right), \quad
\tilde{\gamma}_2(x,y) = \left(\frac{x}{2}+\frac{1}{2},\frac{y}{2}\right).
\]
Let $\alpha_{\theta}$ be a rotation by angle $\theta$.
We define $\gamma_0 = \tilde{\gamma}_0$, $\gamma_1 = \alpha_{-2\pi/3}\circ
\tilde{\gamma}_1$, $\gamma_2 = \alpha_{2\pi/3}\circ \tilde{\gamma}_2$.
We denote by $\Delta \rm{ABC}$ the regular triangle whose vertices
are A, B and C. Then  $\gamma_0(\Delta {\rm P}{\rm Q}{\rm R}) = \Delta {\rm P}{\rm S}{\rm U}$,
$\gamma_1(\Delta {\rm P}{\rm Q}{\rm R}) = \Delta {\rm T}{\rm S}{\rm Q}$ and
$\gamma_2(\Delta {\rm P}{\rm Q}{\rm R}) = \Delta {\rm T}{\rm R}{\rm U}$. 
Let 
$$
K = \bigcap_{n=1}^{\infty}\bigcap_{(j_1,\dots,j_n)\in \Sigma^n}
(\gamma_{j_1}\circ \cdots \circ \gamma_{j_n})({\mathcal T}). 
$$
The compact set $K$ is called a Sierpinski gasket. Then 
$\gamma = (\gamma_0, \gamma_1, \gamma_2)$ is a self-similar map on $K$.
Moreover 
$B_{\gamma}=\{\,{\rm S},{\rm T},{\rm U}\,\}$,
$C_{\gamma}=\{\,{\rm P},{\rm Q},{\rm R}\,\}$. Define a 
 continuous map $h$ by 
\[
  h(x,y) =
\begin{cases}
     &  \gamma_{0}^{-1}(x,y)  \quad (x,y) \in \Delta {\rm P}{\rm S}{\rm U} \cap K,\\
     &  \gamma_{1}^{-1}(x,y)  \quad (x,y) \in \Delta {\rm T}{\rm S}{\rm Q} \cap K, \\
     &  \gamma_{2}^{-1}(x,y)  \quad (x,y) \in \Delta {\rm T}{\rm R}{\rm
 U} \cap K. 
\end{cases}
\]
Then we have that $h({\rm S})={\rm Q}$, $h({\rm U})={\rm R}$, $h({\rm T})={\rm
 P}$, $h({\rm P})={\rm P}$, $h({\rm Q})={\rm R}$ and $h({\rm R})={\rm Q}$.  
Hence  $P_{\gamma}=\{\,{\rm P},{\rm Q},{\rm R}\,\}$ and the self-similar map
$\gamma$ on $K$ satisies the assumption B. 
\end{exam}

We recall a construction of a \cst-correspondence (or Hilbert \cst-bimodule) 
for the self-similar map
$\gamma=(\gamma_0,\dots,\gamma_{N-1})$ in \cite{KW1}.  
Let $A={\rm C}(K)$ be a coefficient algebra. Consider 
a cograph $\C_{\gamma} =\{\,(\gamma_j(y),y)\,|\,
j\in \Sigma,\, y \in K\,\}$. Let $X_{\gamma}={\rm C}(\C_{\gamma})$.
Define left and right $A$-module actions and an $A$-valued inner
product on $X_{\gamma}$ by,  for $f$, $g \in X_{\gamma}$ and $a$, $b \in A$, 
\begin{align*}
  & (a\cdot f\cdot b)(\gamma_j(y),y)=
  a(\gamma_j(y))f(\gamma_j(y),y)b(y) \quad  y \in K, \quad j=0,\dots,N-1, \\
  & (f|g)_A (y) = \sum_{j=0}^{N-1}
  \overline{f(\gamma_j(y),y)}g(\gamma_j(y),y).
\end{align*}
Recall that the "compact operators" $\K(X_{\gamma})$ 
is the \cst-algebra generated by the rank one operators 
$\{\theta_{x,y} \ | \ x,y \in X_{\gamma} \}$, 
where $\theta_{x,y}(z) = x(y|z)_A$ for $z \in X_{\gamma}$. 
When we do stress the role of the module $X$, we write $\theta_{x,y} =
\theta^X_{x,y}$. We denote by $\K(X_{\gamma})$ the set of "compact" operators on
$X_{\gamma}$, and by $\L(X_{\gamma})$ the set of adjointable operators
on $X_{\gamma}$. We define  an $*$-homomorphism $\phi$ of $A$ to
$\L(X_{\gamma})$ by $\phi(a)f=a\cdot f$. 


\begin{lemma} \cite{KW1} 
Let $\gamma=(\gamma_0,\dots,\gamma_{N-1})$ be a self-similar map 
on a compact set $K$. Then 
$X_{\gamma}$ is an $A$-$A$ correspondence and  full as a 
right Hilbert module. 
\end{lemma}

We write by $\O_{\gamma}= \O_{\gamma}(K)$ the Cuntz-Pimsner \cst -algebra 
(\cite{Pi})
associated with the \cst-correspondence $X_{\gamma}$. We call it 
the Cuntz-Pimsner algebra $\O_{\gamma}$ 
associated with a self-similar map ${\gamma}$. 
We usually identify $i(a)$ with $a$ in $A$.  
We also identify $S_{\xi}$ with $\xi \in X$ and simply write $\xi$
instead of $S_{\xi}$. The {\it gauge action}
$\alpha : {\mathbb R} \rightarrow \Aut \ {\mathcal O}_{\gamma}$ is 
defined by $\alpha_t(S_{\xi}) = e^{it}S_{\xi}$ for $\xi\in X_{\gamma}$ 
and $\alpha_t(a)=a$ 
for $a\in A$.

\begin{thm} \cite{KW1} 
Let  $\gamma$  be a self-similar map  on a compact metric space $K$.  
If $(K,\gamma)$ satisfies the  open set condition, then 
the associated Cuntz-Pimsner algebra
$\O_{\gamma}$ is simple and purely infinite.  
\end{thm}

For example, if $\gamma$ is a self-similar map associated with the 
tent map or the 
Sierpinski gasket, then $\O_{\gamma}$ is simple and purely infinite.  

For the 
$n$-times inner tensor product $X_{\gamma}^{\otimes n}$ of $X_{\gamma}$, 
we denote by $\phi_n$ the left module action of $A$ on $X_{\gamma}^{\otimes n}$.
Define 
$$
\F^{(n)} = A \otimes I + \K(X_{\gamma})\otimes I + 
\K(X_{\gamma}^{\otimes 2})\otimes I + 
 \cdots + \K(X_{\gamma}^{\otimes n}) \subset \L(X_{\gamma}^{\otimes n})
$$
Then $\F^{(n)}$ is embedded to  $\F^{(n+1)}$ by $T \mapsto T\otimes I$ for 
$T \in \F^{(n)}$. Let 
$\F^{(\infty)} = \overline{\bigcup_{n=0}^{\infty}\F^{(n)}}$.
It is important to recall that Pimsner \cite{Pi} shows that 
we can identify $\F^{(n)}$ with the  
\cst-subalgebra of  $\O_{\gamma}$ generated by $A$ and $S_xS_y^*$ 
for $x,y \in X^{\otimes k}$, $k=1,\dots, n$ 
under identifying $S_xS_y^*$  with 
$\theta_{x,y}$. Therefore we see that the inductive limit algebra 
$\F^{(\infty)}$ can be identified with 
the fixed point subalgebra $\O_{\gamma}^{\mathbb T}$ of $\O_{\gamma}$
under the gauge action. We call it the {\it core} of $\O_{\gamma}$.
We also call $\F^{(n)}$ the {\it finite core} or the finite $n$-th core
of $\O_{\gamma}$.

\section
{Matrix representation over the coefficient algebra of cores}
\label{sec:Matrix}

For a Hilbert right $A$-module $X$, 
we denote by $\K_0(X)$ the set of  finite sums of "rank one "
operators $\theta_{x,y}$ for $x,y \in X$. If $X$ is finitely generated 
projectile module, then $\K_0(X) = \K(X)= \L(X)$. 
If a self-similar map $\gamma=(\gamma_0,\dots,\gamma_{N-1})$ 
has a branched point, then the Hilbert right $A$-module 
$X_{\gamma}$ is not finitely generated projectile module and 
$\K(X_{\gamma})\not= \L(X_{\gamma})$. But 
if the self-similar map $\gamma$ satisfies Assumption B, it is shown 
in \cite{KW5} that $X_{\gamma}$ is near to a finitely generated
projectile module in the following sense: 
$\K_0(X_{\gamma}) = \K(X_{\gamma})$. 
Moreover $\K(X_{\gamma})$ is realized as 
a subalgebra of the full matrix algebra 
$M_N(A) = C(K, M_N({\mathbb C})$ over $A = C(K)$  
consisting of matrix valued functions $f$ on $K$ such that 
their scalar matrices $f(c)$ live in certain  subalgebras 
$Q_{\gamma}(c)M_N({\mathbb C})Q_{\gamma}(c)$ cutted down by projections 
$Q_{\gamma}(c)$ 
for each $c$ in the finite set $C_{\gamma}$ and live in 
the full matrix algebra 
$M_N({\mathbb C})$ for other $c \notin C_{\gamma}$. Therefore 
$\K(X_{\gamma})$ has certain informations on the sigularity 
structure of branched points. 

We present only the outline of the matrix representation 
over the coefficient algebra of the finite
core of $\O_{\gamma}$ under Assumption B.  We refer the detail to the 
section two of \cite{KW5}.  

Let $Y_{\gamma}:=A^N$ be a free module over $A = {\rm C}(K)$. 
Then $\L(Y_{\gamma})$ 
is isomorphic to $M_N(A)$. 
Recall that  left and right $A$-module actions and an $A$-inner product 
on $Y_{\gamma}:=A^N$ are given as follows: 
\begin{align*}
& (a \cdot f \cdot b)_i(y) = a(\gamma_i(y))f_i(y)b(y),  \\
& (f|g)_A(y) = \sum_{i=0}^{N-1} \overline{f_i(y)}g_i(y),
\end{align*}
where $f=(f_0,\dots,f_{N-1})$, $g=(g_0,\dots,g_{N-1}) \in Y_{\gamma}$ and
$a$, $b \in A$.  Thus 
$Y_{\gamma}$ is an $A$-$A$ correspondence and $Y_{\gamma}$
is a finitely generated projective right module over $A$.

Let $\mathcal A$ be a subalgebra of 
$\L(Y_{\gamma}) = M_N(A) = C(K,M_N({\mathbb C}))$.  
For $z \in K$, the fiber ${\mathcal A}(z)$ of $\mathcal A$ 
at $z$ is defined by 
\[
{\mathcal A}(z) = \{T(z) \in M_N({\mathbb C}) \ | 
\ T \in {\mathcal A}  \subset C(K,M_N({\mathbb C})) \}.
\]  

We define a subset $Z_{\gamma}$ of $Y_{\gamma}$ by
\begin{align*}
& Z_{\gamma}:= \{\,f = (f_0,\dots,f_{N-1}) \in  A^N\,  |\, \\
& \text{ for any } 
 c \in C_{\gamma}, b \in B_{\gamma} \text { with } h(b) = c, \  
  f_{j(b,k)}(c)=f_{j(b,k')}(c) \quad 0\le k, k' \le
e_{b}-1 \,   \}.   
\end{align*}

Thus $i$-th component $f_i(c)$ of the 
vector $(f_1(c),\dots,f_N(c)) \in {\mathbb C} ^N$ is equal to the 
$i'$-th component $f_{i'}(c)$ of it for any $i,i'$ in the same index subset 
$$
\{ j \in \Sigma\ |\ b =\gamma_{j}(c)\} = 
 \{j(b,0),j(b,1),\dots, j(b,e_b-1)\}
$$
for each  $b \in B_{\gamma}$. 

In order to get the idea of rather complicated but important 
notations below, 
we consider a simplified example: 
Assume  that $N = 5, c \in C_{\gamma}$ and 
$h^{-1}(c) = \{b_1, b_2\} \subset B_{\gamma}$, 
\[
b_1 =  \gamma_0(c)= \gamma_1(c), 
\ \ b_2 = \gamma_2(c)=  \gamma_3(c)= \gamma_4(c). 
\]
that is, 
\[
b_1 \; \overset{\gamma_0, \gamma_1 }\Longleftarrow \; 
c \; \overset{\gamma_2, \gamma_3, \gamma_4}\Longrightarrow \; 
b_2 .
\]

Consider the following  degenerated subalgebra ${\mathcal A}$ 
of a full matrix algebra $M_5({\mathbb C})$: 
\[
{\mathcal A} = \{a = (a_{ij}) \in M_5({\mathbb C}) \ | \ 
a_{0j} = a_{1j},\  a_{i0} = a_{i1}, \ a_{2j} = a_{3j}=a_{4j},\  
a_{i2} = a_{i3} = a_{i4} \}. 
\]
Then 
\[
{\mathcal A} 
= \{ \begin{pmatrix}
  a & a & b & b & b  \\
  a & a & b & b & b  \\
  c & c & d & d & d  \\
  c & c & d & d & d  \\
  c & c & d & d & d  \\
\end{pmatrix} 
\ | \ a,b,c,d \in {\mathbb C} \} 
\]
is isomorphic to $M_2({\mathbb C})$. 
Consider the subspace 
\[
W = \{(x,x,y,y,y)\in {\mathbb C}^5 \ | 
\ x \in {\mathbb C}, y \in {\mathbb C}\}
\]
of ${\mathbb C}^5$. 
Let $u_1 = \frac{1}{\sqrt 2} \ ^t(1,1,0,0,0) \in W$ and  
$u_2 = \frac{1}{\sqrt 3} \ ^t(0,0,1,1,1) \in W$.  
Then $\{u_1,u_2\}$ is a basis of $W$. Let $Q_{\gamma}(c)$ 
be the projection onto the the subspace $W$. Then 
\[
{\mathcal A} = Q_{\gamma}(c)M_5({\mathbb C})Q_{\gamma}(c) 
\cong M_2({\mathbb C}). 
\]
Moreover 
$\{\theta^{W}_{u_i,u_j}\}_{i,j =1,2}$ is a 
matrix unit of ${\mathcal A}$ and 
\[
{\mathcal A} 
= \{ \sum_{i,j=1}^2 a_{ij} \theta^{W}_{u_i,u_j} 
\ | \  a_{ij} \in {\mathbb C}\}
= L(W). 
\]

In general, 
$w_c:= \dim (Z_{\gamma}(c))$ is equal to the cardinality  of $h^{-1}(c)$ 
without counting multiplicity. We can take the following basis 
$\{\,u^c_i\,\}_{i=1,\dots,w_c}$ of the fiber 
$W:= Z_{\gamma}(c) \subset {\mathbb C}^N$: 
Rename $h^{-1}(c) = \{b_1, \dots, b_{w_c}\}$. 
Then the $j$-th component of the vector $u^c_i$ is equal to 
$\frac{1}{\sqrt{e_{b_i}}}$ if 
$j \in \{ j \in \Sigma\ |\ b_i =\gamma_{j}(h(b_i))\} = 
 \{j(b_i,0),j(b_i,1),\dots, j(b_i,e_{b_i}-1)\}$ 
and is equal to 0 if $j$ is otherwise. 

Let $Q_{\gamma}(c)$ 
be the projection onto the the subspace $W$. Then 
\[
\L(W) = Q_{\gamma}(c)M_N({\mathbb C})Q_{\gamma}(c) 
\cong M_{w_c}({\mathbb C}). 
\]

Then $Z_{\gamma}$ is also an $A$-$A$ correspondence
with the $A$-bimodule structure and the $A$-valued inner product 
inherited from $Y_{\gamma}$.  
For $\xi \in X_{\gamma}$, we define $\varphi(\xi) \in Z_{\gamma}$ by 
 \[
  \varphi(\xi)(y) = (\xi(\gamma_0(y),y),\dots,\xi(\gamma_{N-1}(y),y)) \quad 
   y \in K.  
 \]
By Lemma 5 in \cite{KW5}, the \cst-correspondences $X_{\gamma}$ and
$Z_{\gamma}$ are isomorphic under $\varphi$.  
Identifying through  the map $\varphi$, 
we may regard $\K(X_{\gamma})$ as a closed subset
of $\L(Y_{\gamma}) = \L(A^N) = {\rm C}(M_{N}(\comp))$.

For a Hilbert $A$-module $W$, we denote by 
$\K_0(W)$ the set of ``finite rank operators'' 
(i,e, finite sum of rank one operators) on $W$, that is, 
\[
\K_0(W) = \{ \sum_{i=1}^n \theta^{W}_{x_i,y_i} \ | 
\ n \in {\mathbb N}, x_i,y_i \in W \}.
\]

As in \cite{KW5}, we define the following subalgebra $D^{\gamma}$ of  
$M_N(C(K))= C(K, M_N({\mathbb C}))$:  
\begin{align*}
  D^{\gamma} = & \{\, a = [a_{ij}]_{i,j}\in M_N(A)
= C(K, M_N({\mathbb C})) \,|\,
\text{ for } c \in C_{\gamma}, b \in B_{\gamma} 
\text{ with } h(b)=c,  \\
& a_{j(b,k),i}(c)=a_{j(b,k'),i}(c)
  \quad 0 \le k,k'\le e_{b}-1, 0 \le i \le N-1 \\
&  a_{i,j(b,k)}(c)=a_{i,j(b,k')}(c)
  \quad 0 \le k,k'\le e_{b}-1, 0 \le i \le N-1
\}, 
\end{align*}

The algebra $D^{\gamma}$ is a closed *subalgebra of 
$M_N(A)=\K(Y_{\gamma})$ and its fiber $D^{\gamma}(c)$ at $c$ 
is described as 

$$
D^{\gamma}(c) = Q_{\gamma}(c)M_N({\mathbb C})Q_{\gamma}(c)
\cong M_{w_c}({\mathbb C}), 
$$
where $Q_{\gamma}(c)$ is the projection of ${\mathbb C}^N$ 
onto the subspace $W = Z_{\gamma}(c)$ and $w_c = \dim Z_{\gamma}(c)$.

For $f = (f_0,\dots,f_{N-1}) \in  Z_{\gamma}$ and 
$g=(g_0,\dots,g_{N-1}) \in
Z_{\gamma}$, 
the rank one operator
$\theta_{f,g}^{Y_{\gamma}} \in \L(Y_{\gamma})$ 
is in $D^{\gamma}$ and represented by the operator matrix 
  \[
   \theta_{f,g}^{Y_{\gamma}} = [f_i\overline{g}_j]_{ij} \in M_N(A),   
  \]
by Lemma 7 in \cite{KW5}. 
Let $\K(Z_{\gamma} \subset Y_{\gamma}) \subset \L(Y_{\gamma})$  be the 
norm closure of 
\[
\K_0(Z_{\gamma} \subset Y_{\gamma}) 
: = \{ \sum_{i=1}^n \theta^{Y_{\gamma}}_{x_i,y_i} \in \L(Y_{\gamma}) \ | 
\ n \in {\mathbb N}, x_i,y_i \in Z_{\gamma} \}.
\]

The following is Lemma 9 in \cite{KW5}, and essential for matrix
representation of finite cores.  

\begin{lemma}\label{lem:finite-rank} 
Let $\gamma$ be a self-similar map on a compact metric space $K$ and 
satisfies Assumption B.  
Then $\K_0(X_{\gamma})=\K(X_{\gamma})$, 
$\K_0(Z_{\gamma})=\K(Z_{\gamma})$ and 
$\K_0(Z_{\gamma} \subset Y_{\gamma})= \K(Z_{\gamma} \subset Y_{\gamma})  
= D^{\gamma} \subset M_N(A)$.  
\end{lemma}

For any $T \in \K(Z_{\gamma} \subset Y_{\gamma})$, we have 
$T(Z_{\gamma}) \subset Z_{\gamma}$ and by Lemma 8 in \cite{KW5}, 
the restriction map 
\[
\delta :\K(Z_{\gamma} \subset Y_{\gamma})\ni T 
\rightarrow T|_{Z_{\gamma}} \in \K(Z_{\gamma}) 
\]
is an onto *isomorphism such that 
\[
\delta(\sum_{i=1}^n \theta^{Y_{\gamma}}_{x_i,y_i}) = 
\sum_{i=1}^n \theta^{Z_{\gamma}}_{x_i,y_i}. 
\]
Then we have an isometric *-homomorphism $\delta^{-1}$ from 
$\K(Z_{\gamma})$ to $\K({Z_{\gamma} \subset Y_{\gamma}})$ 
satisfying 
\[
 \delta^{-1}(\sum_{i=1}^n \theta^{Z_{\gamma}}_{x_i,y_i}) = 
\sum_{i=1}^n \theta^{Y_{\gamma}}_{x_i,y_i}.  
\]

We can define an injective *-homomorphism  $\Omega^{(1)}$ of  
$\K(X_{\gamma})$ to $\K(Y_{\gamma})$ such that 
\[
  \Omega^{(1)}(\theta^{X_{\gamma}}_{f,g})
     = \theta^{Y_{\gamma}}_{\varphi(f),\varphi(g)}.  
\]
Then by Lemma \ref{lem:finite-rank}, the range of 
$\Omega^{(1)}$ is $D^{\gamma}$.  

Fix a natural number $n$. We shall describe  the matrix representation of 
$\K(X_{\gamma}^{\otimes n})$ for  $n$-times tensor product 
$X_{\gamma}^{\otimes n}$. 
We consider the composition of self-similar maps $\gamma$.  
Put $\gamma^n = \{ \gamma_{i_n}\circ \cdots \circ
\gamma_{i_1}\}_{(i_1,\dots,i_n) \in \Sigma^n}$.  
Then $\gamma^n$ is a self-similar map on
the same compact metric space $K$. For the description of the branched
points for $\gamma^n$, the following Lemmas are given in \cite{KW5}.  

\begin{lemma} \cite{KW5}
\label{lem:branch} 
Let  $\gamma$  be a self-similar map  on a compact metric space $K$ and 
satisfies Assumption B. 
Then $C_{\gamma^n}$ and $B_{\gamma^n}$ are finite sets and 
$C_{\gamma^n} \subset C_{\gamma^{n+1}}$ for each $n= 1,2,3,\dots$. 
The set of branched points $B_{\gamma^n}$ is given by
\[
  B_{\gamma^n} = \{\,(\gamma_{j_k}\circ \cdots \circ
 \gamma_{j_1})(b)\,|\, b \in B_{\gamma},\, (j_1,\dots,j_k) \in
  \Sigma^k, \, 0 \le k \le n-1 \, \}.
\]
Moreover, if $\gamma_{i_n}\circ \cdots \circ \gamma_{i_1}(c)
=\gamma_{j_n}\circ \cdots \circ \gamma_{j_1}(c)$ and 
$(i_1,\dots,i_n) \ne (j_1,\dots,j_n)$, then there exists unique $1 \le s
\le n$ such that $i_s \not= j_s$ and $i_{p}=j_{p}$ for $p \ne s$.
\end{lemma}

And  the following lemma holds for $C_{\gamma^n}$:  

\begin{lemma} \cite{KW5}  \label{lem:C_gamma}
Let $a \in K$.  Then $a \in C_{\gamma^n}$ if and only if there exist
 $0 \le p \le n-1$ and $(i_1,i_2,\dots,i_p) \in \Sigma^{p}$ such that 
$\gamma_{i_{p}\circ \cdots \circ  \gamma_{i_1}}(a) \in C_{\gamma}$.  
 \end{lemma}

We denote by $X_{\gamma^n}$ the $A$-$A$ correspondence for $\gamma^n$.
By Proposition 2.2 in \cite{KW1}, $X_{\gamma}^{\otimes n}$ and $X_{\gamma^n}$
are isomorphic as $A$-$A$ correspondence.  
Using it, we can get similar results for $X_{\gamma}^{\otimes n}$ and  
$Z_{\gamma}^{\otimes n}$.  For example, $X_{\gamma}^{\otimes n}$ is
isomorphic to a closed  submodule $Z_{\gamma^n}$ of $A^{N^n}$.  

For $\gamma^n$, we define a subset $D^{\gamma^n}$ of $M_{N^n}(A)$
as in the case of $\gamma$.  We also consider 
 $C_{\gamma^n}$ instead of $C_{\gamma}$.
We use the same notation $e_b$ for $b \in B_{\gamma^n}$ with $h^n(b)=c$
and $\{\,j(b,k)\,|\,0 \le k \le e_b-1\,\}$
for $\gamma^n$ as in $\gamma$ if there occurs no trouble.  Let 
\begin{align*}
    D^{\gamma^n}
  = \{ \,[a_{ij}]_{ij}\in M_{N^n}(A)  & \,|\,  
\text{ for $c \in C_{\gamma^n}$, $b \in B_{\gamma^n}$ with $h^n(b)=c$, } \\ 
& \text{ $a_{j(b,k),i}(c) = a_{j(b,k'),i}(c), \ \   
 a_{i,j(b,k)}(c) = a_{i, j(b,k')}(c)$  } \\
& \text{ for all $0\le k,k'\le e_b-1$, $0 \le i \le N^n-1$ }
  \, \}.       
\end{align*}

Then $\K(X_{\gamma}^{\otimes n})$ and 
$\K(Z_{\gamma}^{\otimes n})$ are  isomorphic, and $\K(Z_{\gamma}^{\otimes
n})$ is isometrically extended to the subalgebra $D^{\gamma^n}$ of
$Y_{\gamma^n}$.  
As Lemma \ref{lem:finite-rank}, we have the following Proposition:

\begin{prop} \cite{KW5} Let  $\gamma$  be a self-similar map  
on a compact metric space $K$ and 
satisfies Assumption B. Then $\K_0(X_{\gamma}^{\otimes n})$ coincides
with  $\K(X_{\gamma}^{\otimes n})$ and is isomorphic to the  closed
  *subalgebra  $D^{\gamma^n}$ of $M_{N^n}(A)$.
\end{prop}

Let $0 \le r \le n$.  As in the case $r=1$, there exists a matrix 
representation $\Omega^{(r)}$ of  $\K(X_{\gamma}^{\otimes r})$ to
$M_{N^r}(A)$ satisfying 
\[
    \Omega^{(r)}(\theta^{X_{\gamma}^{\otimes r}}_{x,y})
   = \theta^{Y_{\gamma}^{\otimes r}}_{\varphi(x),\varphi(y)}
\]
for $x$, $y \in X_{\gamma}^{\otimes r}$, where $\varphi$ is an
isomorphism from $X_{\gamma}^{\otimes r}$ to $Z_{\gamma}^{\otimes r}$.    

We shall give a matrix representation of the finite core $\F^{(n)}$ in
$M_{N^n}(A)$. Let $n \ge 0$.  Fix $0 \le r \le n$.  Then 
there exist a representation $\Omega^{(n,r)}$ from
$\K(X_{\gamma}^{\otimes r})$ to $\L(Y_{\gamma}^{\otimes n}) =
M_{N^n}(A)$ such that 
\[
  \Omega^{(n,r)}(T) = \Omega^{(r)}(T) \otimes I_{n-r}
\]
for $T \in \K(X_{\gamma}^{\otimes r})$.  

We introduce some useful notations.  We fix $N$.  
For a sequence of indices $(i_1,\dots,i_s) \in \Sigma^{s}$, 
we define a non negative integer $i_1\dots i_s|_{(N,s)}$ by 
$N$-adic expansion 
\[
 i_1\dots i_s|_{(N,s)} = \sum_{j=1}^s i_j N^{s-j}.  
\]
We shall use this $N$-adic expansion as suffixes to refer 
to the entries in a big matrix.

Let $n$ be an integer, $p$ and $q$ be non negative integers with 
$n=p+q$.  For $T \in M_{N^q}(\comp)$ and $0 \le i \le N^{p}-1$, 
define $\pi_{q,i}(T) \in M_{N^n}(\comp)$ by
\[
  (\pi_{q,i}(T))_{a,b} 
     = \begin{cases}  &T_{a-iN^q, b-iN^q}   \quad (iN^q \le
	a,b \le (i+1)N^q -1)
  \\
    & 0 \quad ({\rm otherwise})
       \end{cases}.  
\]
Thus $(\pi_{q,i}(T))_{a,b}$ is a block diagonal matrix such that 
the only $i$-th block is $T$ and the other blocks are $0$. 

For mutually different
integers $0\le i_k \le N^p-1$ $(k=1,\dots,r)$, define 
a block diagonal matrix such that 
the  $i_1, \dots, i_r $-th blocks are the same $T$ 
and the other blocks are $0$ by 

\[
 \pi_{q,i_1,\dots, i_r}(T) = \sum_{k=1}^r \pi_{q,i_k}(T).  
\]

Let $0 \le i,j \le N^p-1$, $i \ne j$.  We also define $\sigma_{q,i,j}(T)
\in M_{N^n}(\comp)$ by 
\[
 (\sigma_{q,i,j}(T))_{a,b}
    = 
\begin{cases} 
  &  T_{a - iN^q,b -iN^q}  \quad  iN^q \le a,b \le (i+1)N^{q}-1 \\ 
  &  T_{a - iN^q,b -jN^q } \quad  iN^q \le a \le (i+1)N^{q}-1, \,\,
                   jN^q \le b \le (j+1)N^{q}-1      \\
  &  T_{a  -jN^q,b-iN^q}  \quad jN^q \le a \le (j+1)N^{q}-1,\,\,
                   iN^q \le b \le (i+1)N^{q}-1 \\
  &  T_{a  -jN^q,b-jN^q}  \quad jN^q \le a,b \le (j+1)N^{q}-1 \\
  & 0 \quad {\rm otherwise}.  
\end{cases}
\]

Thus $(\sigma_{q,i,j}(T))_{a,b}$ is a a block matrix such that 
the  $i-i, i-j, j-i, j-j $-th blocks are the same $T$ 
and the other blocks are $0$. 

By Lemma 15 in \cite{KW5}, the natural embedding
\[
 \L(Y_{\gamma}^{\otimes r}) \ni T \mapsto 
T \otimes I_{n-r} \in \L(Y_{\gamma}^{\otimes n}) ={\rm C}(K,M_{N^n}(\comp))
\]
has a matrix representation as a block diagonal matrix for each 
fiber at ${\rm P} \in K$: 
\[
(T \otimes  I_{n-r})({\rm P})
 =  \sum_{(i_1,\dots,i_{n-r}) \in \Sigma^{n-r}}
  \pi_{r,i_1 \dots i_{n-r}}|_{(N,n-r)}(T((\gamma_{i_{n-r}}\circ \cdots
 \circ \gamma_{i_1})({\rm P}))),  
\]  
where $T((\gamma_{i_{n-r}}\circ \cdots \circ \gamma_{i_1})
({\rm P})) \in M_{N^r}(\comp)$
is a fiber of $T$ at 
$(\gamma_{i_{n-r}}\circ \cdots \circ \gamma_{i_1})({\rm P})$ . 
Therefore $(T \otimes  I_{n-r})({\rm P})$ is a block diagonal 
matrix by these fibers. 

Therefore we can describe the form of 
\[
\Omega^{(n,r)} : \K(X_{\gamma}^{\otimes r}) 
\rightarrow \L(Y_{\gamma}^{\otimes n}) = M_{N^n}(A).  
\]
For $T\in \K(X_{\gamma}^{\otimes r})$ $(0 \le r
\le n-1)$, $\Omega^{(n,r)}(T)$ is expressed as:
\[
  \Omega^{(n,r)}(T)({\rm P}) =
  \sum_{(i_1,\dots,i_{n-r}) \in \Sigma^{n-r}}
  \pi_{r,i_1 \dots
  i_{n-r}}|_{(N,n-r)}(\Omega^{(r)}(T)((\gamma_{i_{n-r}}\circ \cdots
 \circ \gamma_{i_1})({\rm P}))
\]
for each ${\rm P} \in K$.  

As in \cite{KW5}, we have  the following Theorem:  

\begin{thm}  \cite{KW5} {\bf (matrix representation of the finite
 core)}
Let  $\gamma$  be a self-similar map  on a compact metric space $K$ and 
satisfies Assumption B. Then there exists an isometric $*$-homomorphism 
$\Pi^{(n)} :  \F^{(n)} \rightarrow M_{N^n}(A)$ such that,  
for $T=  \sum_{r=0}^{n}T_r \otimes I_{n-r} \in \F^{(n)}$ with  
$T_r \in \K(X_{\gamma}^{\otimes r})$,
\[
\Pi^{(n)}(T) =  \sum_{r=0}^n \Omega^{(n,r)}(T_r), 
\]
and if we identify $X_{\gamma}^{\otimes r}$ with 
$Z_{\gamma}^{\otimes r}$, then 
\[
\Omega^{(n,r)}(\theta^{Z_{{\gamma}^r}}_{x,y})
= \theta^{Y_{{\gamma}^r}}_{x,y} \otimes I_{n-r}. 
\]
The image $\Pi^{(n)}(T)$ is independent of the expression of $T=
\sum_{r=0}^{n}T_r \otimes I_{n-r} \in \F^{(n)}$. 

  Moreover the following diagram commutes:
\[
  \begin{CD}
   \F^{(n)} @> {\Pi^{(n)}}>>
    M_{N^n}(A)  \\
   @VVV @VVV \\
   \F^{(n+1)}  @>{\Pi^{(n+1)}}>>  M_{N^{n+1}}(A).
  \end{CD}
\]
In particular the core 
 $\F^{(\infty)}$ is represented in the $M_{N^{\infty}}\otimes A$ as a \cst
-subalgebra.
\end{thm}

We note that for $T \in \K(X^{\otimes r})$ it holds that 
\[
 \Pi^{(n)}(T)({\rm P}) = \sum_{(i_1,\dots,i_{n-r}) \in \Sigma^{n-r}}
 \pi_{r,i_1 \dots i_{n-r}}|_{(N,n-r)}(\Pi^{(r)}(T))((\gamma_{i_{n-r}}\circ \cdots
 \circ \gamma_{i_{1}})({\rm P})).  
\]
Hence  $\Pi^{(n)}(T)({\rm P})$ is a block diagonal matrix consisting of 
$\Pi^{(r)}(T))((\gamma_{i_{n-r}}\circ \cdots \circ \gamma_{i_{1}})({\rm P})$.

\section
{Explicit form of finite cores in terms of singularity}
\subsection
{General calculation}
Let ${\bf \gamma} = (\gamma_0,\dots,\gamma_{N-1})$ be a 
self-similar map on $K$ satisfying the assumption B.  
We shall describe the structure of 
$\Pi^{(n)}(\F^{(n)})({\rm P})$ for ${\rm P} \in
P_{\gamma}$.  Let $0 \le q \le n$, and $p=n-q$.  
The elements of $\Pi^{(n)}(\K(X^{\otimes q}))({\rm P})$ are represented as
block diagonal matrices with blocks which consists  
of matrices in $M_{N^q}(\comp)$.  

\begin{lemma} \label{lem:index}
Let ${\rm P} \in K$.  Assume that $\gamma_{i_k}\circ \cdots \circ
 \gamma_{i_1}({\rm P}) \notin B_{\gamma}$ (for any $k=1,\dots,p-1$).  
If 
\[
\gamma_{i_{p-1}}\circ \gamma_{i_{p-2}}\circ \cdots \circ
 \gamma_{i_1}({\rm P}) = \gamma_{j_{p-1}}\circ \gamma_{j_{p-2}}\circ \cdots \circ
  \gamma_{j_1}({\rm P}),  
\]
then it holds $(i_1,i_2,\dots,i_{p-1})=(j_1,i_2,\dots,j_{p-1})$.  
\end{lemma}

\begin{proof} It is clear fronm the definition of $B_{\gamma}$. 
\end{proof}

\begin{defn} For a point ${\rm R} \in K$ and 
non-negative integers $p$ and $q$, the $p-q$ orbit $Orbit_{p-q}({\rm R})$ 
through  $R$ is defined by 

\begin{align*}
Orbit_{p-q}({\rm R})  
:=  \{ {\rm R}_0 \overset{\gamma_{i_1}}\to {\rm R}_1 
\overset{\gamma_{i_2}}\to {\rm R}_2 \dots \overset{\gamma_{i_{p+q}}}\to 
{\rm R}_{p+q}\ | 
& \ ({\rm R}_0, \dots, {\rm R}_{p+q}) \in K^{p+q+1},\\
& (i_1,i_2,\dots,i_{p+q})\in \Sigma^{p+q}, \ \gamma_{i_p}
({\rm R}_0) = {\rm R} \}
\end{align*}
\end{defn} 

The following proposition shows that $p-q$ orbit $Orbit_{p-q}({\rm R})$ 
through  ${\rm R}$ representes  the branched points structure. And this will 
make the explicit computation of the finite core possible. 

\begin{prop}Suppose the Assmption B . 
For non-negative integers $p$ and $q$, if ${\rm R}$ is a  branched point, 
then any element in the $p-q$ orbit $Orbit_{p-q}({\rm R})$ 
through  ${\rm R}$ satisfies the following: \\
The points ${\rm R}_0,{\rm R}_1, \dots, {\rm R}_{p-1}$ and a finite word 
$(i_1,\dots,i_{p-1})$ are uniquely determined, and 
$i_p$ is $j({\rm R},0),j({\rm R},1), \dots$  or $j({\rm R},{e_{\rm R}-1})$. 
Moreover the number of $Orbit_{p-q}({\rm R})$ is exactly $e_{\rm R} N^q$. 
\end{prop}
\begin{proof}
By the Assmption B, ${\rm R}_k = h^k({\rm R})$ for $k =0,1,\dots,p-1$ and 
$(i_1,\dots,i_{p-1})$ is uniquely determined by Lemma \ref{lem:index}. Then 
the rest is clear. 
\end{proof}

We write the above situation by the following picture: 

\[
{\rm P}={\rm R}_0 \overset{\gamma_{i_1}}\to {\rm R}_1 
\overset{\gamma_{i_2}}\to {\rm R}_2 \dots 
\overset{\gamma_{i_{p-1}}}\to {\rm R}_{p-1}  
\overset{\gamma_{i_{j({\rm R},0)}},\gamma_{i_{j({\rm R},1)}}, \dots, 
\gamma_{i_{j({\rm R},e_{\rm R}-1)}}}
\Longrightarrow {\rm R}_p ={\rm R} 
\begin{matrix}
\ & \nearrow \\
\ & ...      \\
\ & \searrow
\end{matrix}
\begin{matrix}
\ & \nearrow \\
\ & ...      \\
\ & \searrow
\end{matrix}
\]

Let ${\rm P}$ be a point of $K$ and ${\rm R}$ be a branched point. 
Fix a natural number $n$. Then 
the $p-q$ orbit $Orbit_{p-q}({\rm R})$ through  ${\rm R}$
suggests us an important subalgebra $\tilde{{\rm C}}({\rm P},{\rm R},p)$
 of $\Pi^{(n)}(\F^{(n)})({\rm P})$,  
which represents  a building block of the finite core. 

\begin{defn}
Let ${\rm P}$ be a point of $K$ and ${\rm R}$ be a branched point. 
Fix a natural number $n$. Let $p$ be a non-negative integer with 
$h^p({\rm R}) = {\rm P}$ and $q = n-p$. 
Define a subalgebra $\tilde{{\rm C}}({\rm P},{\rm R},p)$ of $\Pi^{(n)}(\F^{(n)})({\rm P})$ by
\begin{align*}
    & \tilde{{\rm C}}({\rm P},{\rm R},p) \\
  = &\{\, \pi_{q,i_1\dots i_{p-1}j({\rm R},0)|_{(N,p)}, \,
i_1\dots i_{p-1}j({\rm R},1)|_{(N,p)},\, \dots, \, 
i_1\dots i_{p-1}j({\rm R},e_{\rm R}-1)|_{(N,p)}
}(A)\,|\, A \in M_{N^{q}}(\comp)\,\} \\
 = & \{\,
\sum_{0 \le l \le e_{\rm R}-1} 
\pi_{q,i_1\dots i_{p-1}j({\rm R},l)|_{(N,p)}}
(A)\,|\, A \in M_{N^{q}}(\comp) 
\,\}.  
\end{align*}
Any operator $T = \sum_{0 \le l \le e_{\rm R}-1} 
\pi_{q,i_1\dots i_{p-1}j({\rm R},l)|_{(N,p)}}
(A)$ in the subalgebra $\tilde{{\rm C}}({\rm P},{\rm R},p)$
is represented as a block diagonal matrix consisting of 
$A \in M_{N^{q}}(\comp)$ 
such that $A$ appears with multiplicity $e_{\rm R}$ 
in the diagonal positions of $i_1\dots i_{p-1} i_p|_{(N,p)}$ with 
$i_p = j({\rm R},0),j({\rm R},1), \dots, j({\rm R},{e_{\rm R}-1})$.
Therefore the subalgebra $\tilde{{\rm C}}({\rm P},{\rm R},p)$ 
can be read off from the singularity structure of branched points. 
\end{defn}

If we fix $n$, then 
the algebra $\tilde{\rm C}({\rm P},{\rm R},p)$ is uniquely determined by
${\rm P} \in K, {\rm R} \in
B_{\gamma}$ and $p$.  Moreover $\tilde{\rm C}({\rm P},{\rm R},p)$ is isomorphic to 
$M_{N^q}(\comp)$ as a \cst -algebra with $q = n-p$.  

The following Proposition characterizes matrices in
$\Pi^{(m)}(\K(X^{\otimes m}))({\rm P})$.  

\begin{prop} Let $m \ge 1$ be an integer.  
If ${\rm P} \notin C_{\gamma^m}$, then $\Pi^{(m)}(\K(X^{\otimes m}))({\rm P})
 =M_{N^m}(\comp)$. If ${\rm P} \in C_{\gamma^m}$, then  
a matrix $T \in M_{N^m}(\comp)$ is contained in $\Pi^{(m)}(\K(X^{\otimes m}))({\rm P})$
if and only if $T$ satisfies the following: \\
For any $1 \le p \le m$ and ${\rm R} \in B_{\gamma}$ with  $h^p({\rm R})={\rm P}$, take  $(i_1,\dots,i_{p-1}) \in \Sigma^{p-1}$, which is uniquely determined   as above. 
\[
T_{i_1\dots i_{p-1} j({\rm R},t) \tilde{i}_{p+1}\dots
 \tilde{i}_m|_{(N,m)}, 
\overline{j}_1\overline{j}_2 \dots \overline{j}_m{}|_{(N,m)}}
=  T_{i_1\dots i_{p-1} j({\rm R},t') \tilde{i}_{p+1}\dots
 \tilde{i}_m|_{(N,m)}, 
\overline{j}_1\overline{j}_2 \dots \overline{j}_m{}|_{(N,m)}}
\]
\[
T_{\overline{i}_1 \overline{i}_2 \dots \overline{i}_{m}|_{(N,m)}, 
 i_1 \dots i_{p-1} j({\rm R},s) \tilde{j}_{p+1}\dots
 \tilde{j}_{m}|_{(N,m)}}  
 =  T_{\overline{i}_1 \overline{i}_2 \dots \overline{i}_{m}|_{(N,m)}, 
 i_1 \dots i_{p-1} j({\rm R},s') \tilde{j}_{p+1}\dots 
 \tilde{j}_{m}|_{(N,m)}}, 
\]

for arbitrary $0 \le t,t',s,s' \le e_{\rm R}-1$ and arbitrary 
$(\tilde{i}_{p+1},\dots,\tilde{i}_{m})$, 
$(\tilde{j}_{p+1},\dots,\tilde{j}_{m}) \in \Sigma^{m-p}$, 
$(\overline{i}_1,\dots,\overline{i}_m)$,
$(\overline{j}_1,\dots,\overline{j}_m) \in \Sigma^{m}$. \\
Moreover
\[
\Pi^{(m)}(\K(X^{\otimes m}))({\rm P}) 
= Q_{\gamma^m}({\rm P})M_{N^m}(\comp)Q_{\gamma^m}({\rm P}), 
\]
where $Q_{\gamma^m}({\rm P}) \in M_{N^m}(\comp)$ is the projection 
onto the subspace $Z_{\gamma^m}({\rm P}) \subset {\mathbb C}^{N^m}$. 
 \end{prop}

\begin{proof} It follows from the form of $D^{\gamma^m}$ and the
 description of $B^{\gamma^m}$.  
\end{proof}

\begin{lemma} \label{lem:compact}
Let ${\rm P} \in P_{\gamma}$ and $n = p + q$. 
Assume that $(i_1,i_2,\dots,i_p)$ satisfies 
$\gamma_{i_k}\circ \cdots \circ
 \gamma_{i_1}({\rm P}) \notin B_{\gamma}$ for any $k=1,\dots,p$.  
Then it holds that 
\[
  \pi_{q,i_1\dots i_{p}|_{(N,p)}}(\Pi^{(q)}(\K(X^{\otimes
 q}))(\gamma_{i_p} \circ
 \dots \circ \gamma_{i_1}({\rm P}))) \subset   \Pi^{(n)}(\K(X^{\otimes
 n}))({\rm P}).  
\]
\end{lemma}

\begin{proof} We shall show the left hand side is contained in the right hand
 side.  Put ${\rm Q}=\gamma_{i_p} \circ \cdots \circ \gamma_{i_1}({\rm P})$.  
Then ${\rm Q} \notin B_{\gamma}$.  $\Pi^{(q)}(\K(X^{\otimes q}))({\rm Q})$ is a
 subalgebra of $M_{N^q}(\comp)$.  A matrix $T \in M_{N^q}(\comp)$ is
 contained in $\Pi^{(q)}(\K(X^{\otimes q}))({\rm Q})$ if and only if the 
following holds: 

For any $r$ and ${\rm R}$ with 
$1 \le r \le q$ and ${\rm R}=\gamma_{i_{p+r}}\circ \cdots \gamma_{p+1}({\rm Q}) \in B_{\gamma}$, 
\begin{equation} \label{eq:row}
 T_{i_{i_{p+1}}\dots i_{p+r-1}j({\rm R},t)\tilde{i}_{p+r+1}\dots
 \tilde{i}_{n}|_{(N,q)}, \overline{j}_{p+1}\dots
 \overline{j}_{n}|_{(N,q)}}
 = T_{i_{i_{p+1}}\dots i_{p+r-1}j({\rm R},t')\tilde{i}_{p+r+1}\dots
 \tilde{i}_{n}|_{(N,q)}, \overline{j}_{p+1}\dots
 \overline{j}_{n}|_{(N,q)}}, 
\end{equation}
and 
\begin{equation} \label{eq:column}
 T_{\overline{i}_{p+1}\dots \overline{i}_{n}|_{(N,q)},
  i_{p+1}\dots i_{p+r-1}j({\rm R},s)\tilde{j}_{r+1}\dots
 \tilde{j}_{n}|_{(N,q)}} 
 = T_{\overline{i}_{p+1}\dots \overline{i}_{n}|_{(N,q)},
  i_{p+1}\dots i_{p+r-1}j({\rm R},s')\tilde{j}_{r+1}\dots
 \tilde{j}_{n}|_{(N,q)}},  
\end{equation}
for each $0 \le t,t',s,s' \le e_{{\rm R}}-1$ and for each
 $(\tilde{i}_{p+r+1},\dots,\tilde{i}_{n})$,
 $(\tilde{j}_{p+r+1},\dots,\tilde{j}_{n}) \in \Sigma^{q-r}$ , 
 $(\overline{i}_1,\dots,\overline{i}_{q})$,
 $(\overline{j}_1,\dots,\overline{j}_q) \in \Sigma^{q}$.  

Now assume that  $T \in M_{N^q}(\comp)$ satisfies the above condition 
of compact algebra. Put 
\[
   S = \pi_{q,i_1 \dots i_p}|_{(N,p)}(T) \in M_{N^n}(\comp).  
\]
We add $(i_1,\dots,i_p)$ to the row and the column of the condition 
\eqref{eq:row}, then we get the following:
\begin{align*}
 &  S_{i_1 \dots i_p i_{p+1} \dots i_{p+r-1}j({\rm R},t)\tilde{i}_{p+r+1}\dots
 \tilde{i}_{n}|_{(N,n)} , i_1 \dots i_p \overline{j}_{p+1} \dots
 \overline{j}_{n}|_{(N,n)}}  \\
 = & S_{i_1 \dots i_p i_{p+1} \dots i_{p+r-1}j({\rm R},t')\tilde{i}_{p+r+1}\dots
 \tilde{i}_{n}|_{(N,n)}, i_1 \dots i_p \overline{j}_{p+1} \dots
 \overline{j}_{n}|_{(N,n)}}.  
\end{align*}
To show that $S$ is in $\Pi^{(n)}(\K(X^{\otimes n}))({\rm P})$, it is 
sufficient to show that the 
equality holds for arbitrary column.  
For $(\overline{j}_1,\dots,\overline{j}_p) \ne (i_1,\dots,i_p)$, both
\[
 S_{i_1 \dots i_p i_{p+1} \dots i_{p+r-1}j({\rm R},t)\tilde{i}_{p+r+1}\dots
 \tilde{i}_{n}|_{(N,n)}, \overline{j}_1 \dots \overline{j}_p
 \overline{j}_{p+1} \dots \overline{j}_{n}|_{(N,n)}},  
\]
and 
\[
 S_{i_1 \dots i_p i_{p+1} \dots i_{p+r-1}j({\rm R},t')\tilde{i}_{p+r+1}\dots
 \tilde{i}_{n}|_{(N,n)}, \overline{j}_1 \dots \overline{j}_p
 \overline{j}_{p+1} \dots \overline{j}_{n}|_{(N,n)}} 
\]
are $0$, and the equality holds.  Then if ${\rm R} \in B_{\gamma}$,
 $h^{p+r}({\rm R})={\rm P}$, and $h({\rm R})=\gamma_{i_{p+r-1}}\circ \cdots
 \gamma_{i_1}({\rm P})$, it holds
\begin{align*}
&  S_{i_1 \dots i_pi_{p+1}\dots i_{p+r-1}j({\rm R},t)\tilde{i}_{p+r+1}\dots
 \tilde{i}_{n}|_{(N,n)},\overline{j}_1 \dots
 \overline{j}_p\overline{j}_{p+1}\dots \overline{j}_{n}|_{(N,n)}} \\
 = &  S_{i_1 \dots i_pi_{p+1}\dots i_{p+r-1}j({\rm R},t')\tilde{i}_{p+r+1}\dots
 \tilde{t}_{n}|_{(N,n)},\overline{j}_1 \dots
 \overline{j}_p\overline{j}_{p+1}\dots \overline{j}_{n}|_{(N,n)}}, 
\end{align*}
for each $0 \le t,t' \le e_{\rm R}-1$, for each 
 $(\tilde{i}_{p+r+1},\dots,\tilde{i}_{n}) \in \Sigma^{q-r}$, 
and for each $(\overline{j}_1,\dots,\overline{j}_{n}) \in \Sigma^n$.  
The argument also holds if we interchange columns and rows.  

Therefore 
\[
 S = \pi_{q,i_1 \dots i_p|_{(N,p)}}(T) \in \Pi^{(n)}(\K(X^{\otimes
 n})).  
\]
\end{proof}

If a block matrix $T$ with the following form 
\[ 
T = 
\begin{pmatrix}
  A & 0 & A   \\
  0 & 0 & 0   \\
  A & 0 & A   \\
\end{pmatrix} 
\]
is known to be block diagonal, then 
$T$ must be zero. The following lemma holds 
by a similar reason.

\begin{lemma} 
Let ${\rm P} \in P_{\gamma}$.  Assume that  
${\rm R} =\gamma_{i_p}\circ \cdots \circ \gamma_{i_1}({\rm P}) \in B_{\gamma}$ 
 for some $1 \le p \le n$. 
Then we have that 
\[
     \tilde{\rm C}({\rm P},{\rm R},p)  \cap \Pi^{(n)}(\K(X^{\otimes n}))({\rm P}) =
 \{\,0\,\}.  
\]
\end{lemma}

\begin{proof}
We chose $T \in \tilde{\rm C}({\rm P},{\rm R},p) \subset M_{N^n}(\comp)$.  
We show that if $T \in \Pi^{(n)}(\K(X^{\otimes n}))({\rm P})$, then $T=0$.  
It suffices to show that ($i_1\dots i_{p-1}j({\rm R},l)|_{(N,p)})$-th element
 of the block diagonal matrix $T$ is zero, 
i.e. it is sufficient to show that 
\[
 T_{i_1\dots,i_{p-1}j({\rm R},l)\tilde{i}_{p+1} \dots \tilde{i}_n|_{(N,n)}, 
  i_1\dots i_{p-1}j({\rm R},l)\overline{j}_{p+1}\dots
 \overline{j}_n|_{(N,n)}}
\]
is $0$ for each $(\tilde{i}_{p+1},\dots,\tilde{i}_{n})$, 
$(\overline{j}_{p+1},\dots,\overline{j}_n) \in \Sigma^{n-p}$.  
Since $T \in \Pi^{(n)}(\K(X^{\otimes n})({\rm P})$, 
it holds 
\begin{align*}
 &T_{i_1\dots,i_{p-1}j({\rm R},l)\tilde{i}_{p+1} \dots \tilde{i}_n|_{(N,n)}, 
  i_1\dots i_{p-1}j({\rm R},l)\overline{j}_{p+1}\dots \overline{j}_n|_{(N,n)}}  \\
= & T_{i_1\dots,i_{p-1}j({\rm R},l')\tilde{i}_{p+1} \dots \tilde{i}_n|_{(N,n)}, 
  i_1\dots i_{p-1}j({\rm R},l)\overline{j}_{p+1}\dots
 \overline{j}_n|_{(N,n)}}.   
\end{align*}
The right hand side is $0$ because $T$ is a block diagonal matrix.  
\end{proof}

\begin{lemma} Assume $1 \le p,p' \le n$, 
$h^p({\rm R})=h^{p'}({\rm R}')={\rm P}$, ${\rm R}$, ${\rm R}' \in B_{\gamma}$.  
If $p \ne p'$, or $p=p'$ and ${\rm R} \ne {\rm R}'$, then  we have that 
\[
\tilde{\rm C}({\rm P},{\rm R},p) \cap \tilde{\rm C}({\rm P},{\rm R}',p') = \{\,0\,\}.  
\]
\end{lemma}
\begin{proof}
Assume $h({\rm R})=\gamma_{i_{p-1}}\circ \cdots \circ \gamma_{i_1}({\rm P})$, 
$h({\rm R}')=\gamma_{\tilde{i}_{p'-1}}\circ \cdots \circ
 \gamma_{\tilde{i}_1}({\rm P})$.  We note that 
$(i_1,\dots,i_{p-1})$ and $(\tilde{i}_1,\dots,\tilde{i}_{p'-1})$
are uniquely determined by ${\rm R}$, ${\rm R}'$, $p$, $p'$.  

First assume $p=p'$ and ${\rm R} \ne {\rm R}'$.  
Then it holds then 
\[
 \{\,(i_1,\dots,i_{p-1},j({\rm R},s))\,|\,0 \le s \le e_{\rm R}-1\,\}
 \cap 
 \{\,(\tilde{i}_1,\dots,\tilde{i}_{p-1},j({\rm R}',s'))\,|\, 0 \le s' \le
 e_{{\rm R}'}-1\,\}
= \emptyset.  
\]
Therefore the conclusion follows.  

Next assume that $p' \le p-1$.  

If $(i_1,\dots,i_{p'-1}) \ne
 (\tilde{i}_{1},\dots,\tilde{i}_{p'-1})$, it holds 
\begin{align*}
 \tilde{\rm C}({\rm P},{\rm R},p) \subset & \pi_{n-p'+1,i_1 \dots
 i_{p'-1}|_{(N,p'-1)}}(M_{N^{n-p'+1}}(\comp)) \\
 \tilde{\rm C}({\rm P},{\rm R}',p') \subset & \pi_{n-p'+1,\tilde{i}_1 \dots
 \tilde{i}_{p'-1}|_{(N,p'-1)}}(M_{N^{n-p'+1}}(\comp)).   
\end{align*}
If $(i_1,\dots,i_{p'-1}) = (\tilde{i}_1,\dots,\tilde{i}_{p'-1})$, 
for $0 \le l \le e_{{\rm R}'}-1$ it holds $\tilde{i}_{p'} \ne j({\rm R}',l)$.  Then 
it holds 
\begin{align*}
 \tilde{\rm C}({\rm P},{\rm R},p) \subset & \pi_{n-p',i_1\dots
 i_{p'-1}i_{p'}|_{(N,p')}}(M_{N^{n-p'}}(\comp)) \\
 \tilde{\rm C}({\rm P},{\rm R}',p) \subset & \sum_{0 \le l \le e_{{\rm R}'}-1}
  \pi_{n-p',i_1 \dots i_{p'-1}j({\rm R}',l)|_{(N,p')}}(M_{N^{n-p'}}(\comp)).  
\end{align*}
Since the non-zero positions are different, 
the conclusion follows.  
\end{proof}

\begin{lemma}  \label{lem:inclusion}
Let ${\rm P} \in P_{\gamma}$.  Assume $1 \le p' < p$ and 
${\rm R}' \in B_{\gamma}$ 
satisfy $h^{p'}({\rm R}')={\rm P}$.  We note that $(i_1,\dots,i_{p'-1})$ 
is uniquely 
determined by $h({\rm R}')
=\gamma_{i_{p'-1}}\circ \cdots \gamma_{i_1}({\rm P}) 
\in C_{\gamma}$.  Then we have that 
\begin{align*}
 & \sum_{(\overline{i}_{p'+1},\dots,\overline{i}_{p})
 \in \Sigma^{p-p'}, \,   
0 \le l \le
 e_{{\rm R}'}-1} 
 \pi_{q,i_1\dots,i_{p'-1} j({\rm R}',l) \overline{i}_{p'+1}\dots
 \overline{i}_p} \\
 &    (\Pi^{(q)}(\K(X^{\otimes q})(\gamma_{\overline{i}_p}\circ \cdots
\circ \gamma_{\overline{i}_{p'+1}} \circ 
   \gamma_{j({\rm R}',l)} \circ \gamma_{i_{p'-1}} \circ \cdots \circ  \gamma_{i_1}({\rm P}))) 
 \subset   \tilde{C}({\rm P},{\rm R}',p').  
\end{align*}
\end{lemma}

\begin{proof}
$\tilde{\rm C}({\rm P},{\rm R}',p')$ in the right hand side is 
\[
\tilde{\rm C}({\rm P},{\rm R}',p') 
 = \{\,\pi_{q',i_1\dots i_{p'-1}j({\rm R}',0)|_{(N,p')},\dots,
   i_1 \dots i_{p'-1}j({\rm R}',e_{{\rm R}'}-1)|_{(N,p')}}(A)\,|\, 
   A \in M_{N^{n-p'}}(\comp) \,\}.  
\]
The left hand side is expressed as follows: 
\begin{align*}
& \sum_{(\overline{i}_{p'+1},\dots,\overline{i}_{p}) \in \Sigma^{p-p'} }
 \sum_{0 \le l \le e_{{\rm R}'}-1} \\
 & \pi_{n-p',i_1 \dots i_{p'-1}j({\rm R}',l)|_{(N,p')}}
  (\pi_{q,\overline{i}_{p'+1} \dots \overline{i}_{p}}|_{(N,p-p')}
 (\Pi^{(q)}(\K(X^{\otimes q})((\gamma_{\overline{i}_p}\circ \cdots
 \circ 
 \gamma_{\overline{i}_{p'+1}}\circ
 \gamma_{j({\rm R}',l)}\circ\gamma_{i_{p'-1}} \circ \cdots \circ 
 \gamma_{i_1})({\rm P})).  
\end{align*}
Then 
\begin{align*}
 &  \sum_{0 \le l \le e_{{\rm R}'}-1} 
  \pi_{n-p',i_1 \dots i_{p'-1}j({\rm R}',l)|_{(N,p')}} \\
 & (\pi_{q,\overline{i}_{p'+1} \dots \overline{i}_{p}}|_{(N,p-p')}
 (\Pi^{(q)}(\K(X^{\otimes q})((\gamma_{\overline{i}_p}\circ \cdots
 \circ 
 \gamma_{\overline{i}_{p'+1}}\circ
 \gamma_{j({\rm R}',l)}\circ\gamma_{i_{p'-1}} \circ \cdots \circ 
 \gamma_{i_1})({\rm P})) 
\end{align*}
is contained in the right hand side for each
$(\overline{i}_{p'+1},\dots,\overline{i}_{p}) \in \Sigma^{p-p'}$.  
\end{proof}

The following thorem shows that the  matrix representation over the coefficient algebra of the $n$-th core is described by the fibres at $P$. They 
consist of two parts, that is,  "the compact operators" and the sum of block diagonal algebras $\tilde{\rm C}({\rm P},{\rm R},p)$, which remind 
the pictures of orbits through a branched point.
The key point of the proof is to understand the two effects 
of the branched points 
to "the compact operators" and the inclusion $\F^{(k)} \subset \F^{(k+1)}$. 
For example, these two are described by the following typical forms 
of matrices 
\[ 
\begin{pmatrix}
  A & 0 & A   \\
  0 & 0 & 0   \\
  A & 0 & A   \\
\end{pmatrix} 
\text{ and } 
\begin{pmatrix}
  B & 0 & 0   \\
  0 & 0 & 0   \\
  0 & 0 & B   \\
\end{pmatrix} 
\]

Therefore it is easy to see that 
any block diagonal algebra $\tilde{\rm C}({\rm P},{\rm R},p)$ 
protrudes from "the compact operators" 
$\Pi^{(n)}(\K(X^{\otimes n}))({\rm P})$. The following theorem means
that they are all which protrude from "the compact operators" 
$\Pi^{(n)}(\K(X^{\otimes n}))({\rm P})$.. 
 
\begin{thm} \label{th:representation}
Let ${\rm P} \in P_{\gamma}$. Then the fiber of $
\Pi^{(n)}(\F^{(n)}) \subset C([0,1], M_{2^n}({\mathbb C}))$ 
at ${\rm P}$ is given by 
\begin{align*}
\Pi^{(n)}(\F^{(n)})({\rm P})
   = &  \Pi^{(n)}(\K(X^{\otimes n}))({\rm P})
  + \left(\bigoplus_{1 \le p \le n,\, {\rm R} \in
 B_{\gamma},\,h^p({\rm R})={\rm P}}\tilde{\rm C}({\rm P},{\rm R},p) \right).   
\end{align*}
The first term and the second term is the sum as vector spaces and not
 as *-algebras.  
\end{thm}
\begin{proof}
We shall show the left hand side is contained in the right hand side. 
Let $0 \le q \le n$.  It is enough to show that 
 $\Pi^{(n)}(\K(X^{\otimes q})({\rm P})$
are contained in the right hand side because 
\[
   \Pi^{(n)}(\F^{(n)})({\rm P}) = \sum_{q=0}^n \Pi^{(n)}(\K(X^{\otimes
 q}))({\rm P}).  
\]
Fix $0 \le q \le n$ and put $p=n-q$.  Then it holds
\[
 \Pi^{(n)}(\K(X^{\otimes q}))({\rm P}) 
   = \sum_{(i_1,\dots,i_{p}) \in \Sigma^{p}} 
   \pi_{q,i_1\dots i_p|_{(N,p)}}(\K(X^{\otimes q})((\gamma_{i_{p}} \circ
 \cdots \circ \gamma_{i_1})({\rm P}))).  
\]

We divide $\Sigma^p$ into three groups as follow:
\begin{align*}
 A^p =& \{\, (i_1,\dots,i_p) \in \Sigma^{p}\,|\, 
   \gamma_{i_k}\circ \dots \circ \gamma_{i_1}({\rm P}) \notin B_{\gamma} 
      \quad \text{ for any } k \ (1 \le k \le p)\,\}, \\
 B^p = & \{\, (i_1,\dots,i_p) \in \Sigma^{p}\,|\, 
    \gamma_{i_{p}}\circ \dots \circ \gamma_{i_1}({\rm P}) \in
 B_{\gamma}\,\}, \\
 C^p = & \{\, (i_1,\dots,i_p) \in \Sigma^{p}\,|\, 
  \gamma_{i_{p'}}\circ \gamma_{i_{p'-1}}\circ \cdots \circ \gamma_{i_1}({\rm P}) \in
 B_{\gamma},
 \quad \text{ for some  } p' \ 
 (1 \le p' \le p-1) \,\}.  
\end{align*}

Let ${\rm Q}= \gamma_{i_p}\circ \dots \circ \gamma_{i_1}({\rm P})$. 
Then $A^p$  corresponds to the case that  
there exists no branched point before ${\rm Q}$.  
$B^p$ corresponds to the case that $Q$ is a branched point.  
$C^p$  corresponds to the case that  
there exists a branched point before ${\rm Q}$. 

First we investigate the case $A^p$.  
By Lemma \ref{lem:compact}, it holds
\[
 \sum_{(i_1,\dots,i_{p}) \in A^{p}} 
   \pi_{q,i_1\dots i_p|_{(N,p)}}\Pi^{(n)}(\K(X^{\otimes
 q})(\gamma_{i_{p}} \circ
 \cdots \circ \gamma_{i_1}({\rm P}))) \subset  \Pi^{(n)}(\K(X^{\otimes
 n}))({\rm P}).  
\]

Next we investigate the case $B^p$.  We assume that ${\rm R}_1$, $\dots$, ${\rm R}_m
\in B_{\gamma}$ satisfy $h^p({\rm R}_k)={\rm P}$ for $k=1,\dots,m$.  Then it holds 
\[
  h({\rm R}_k) = (\gamma_{i^k_1}\circ \cdots \circ \gamma_{i^k_{p-1}})({\rm P}) \quad 
 k=1,\dots,m.  
\]
$(i^k_1,\dots,i^k_{p-1})$ $(k=1,\dots,m)$ are uniquely determined.  
Then the set of indices $B_p$ is described as: 
\[
 B_p =\bigcup_{k=1}^{m}\{\,(i^k_1,\dots,i^k_{p-1},j({\rm R}_k,l))\,|\,
 l=0,\dots,e_{{\rm R}_k}-1\,\}.  
\]
Then it holds 
\begin{align*}
&  \sum_{(i_1,\dots,i_{p}) \in B^{p}} 
   \pi_{q,i_1\dots i_p|_{(N,p)}}\Pi(\K(X^{\otimes q})(\gamma_{i_{p}} \circ
 \cdots \circ \gamma_{i_1}({\rm P})))  \\
 = & \sum_{k=1}^m \sum_{l=0}^{e_{{\rm R}_k}-1} 
  \pi_{q,i^k_1 \dots i^k_{p-1}j({\rm R}_k,l)}(\Pi^{(q)}(\K(X^{\otimes
 q})({\rm R}_k))\\
 = & \{\,\sum_{k=1}^m \sum_{l=0}^{e_{{\rm R}_k}-1} \pi_{q,i^k_1 \dots
 i^k_{p-1}j({\rm R}_k,l)}(A_k)\,|\, A_k \in M_{N^q}(\comp) \,\} \\
 = & \sum_{k=1}^m \tilde{\rm C}({\rm P},{\rm R}_k,p).  
\end{align*}

Last we investigate the case $C^p$.  It holds
\begin{align*}
&  \sum_{(i_1,\dots,i_{p}) \in C^{p}} 
   \pi_{q,i_1\dots i_p|_{(N,p)}}\Pi(\K(X^{\otimes q})(\gamma_{i_{p}} \circ
 \cdots \circ \gamma_{i_1}({\rm P})))  \\
= & \sum_{1 \le p' \le p-1,} \sum_{{\rm R} \in B_{\gamma}, h^{p'}({\rm R})={\rm P},}
    \sum_{0 \le l \le  e_{{\rm R}}-1,} \sum_{(i_{p'+1,\dots,i_{p}}) \in \Sigma^{p-p'}} \\
 &     \pi_{q,i_1 \dots i_{p'-1}j({\rm R},l)i_{p'+1 \dots i_p|_{(N,p)}}}
  (\Pi^{(q)}(\K(X^{\otimes q}))(\gamma_{i_p}\circ \cdots \circ
 \gamma_{i_{p'+1}}({\rm R}))) \\
 &  \subset   \sum_{1 \le p' \le p-1,} \sum_{{\rm R} \in B_{\gamma},
 h^{p'}({\rm R})={\rm P}} \tilde{\rm C}({\rm P},{\rm R},p').  
\end{align*}
The last inclusion follows form Lemma \ref{lem:inclusion}.  

The right hand side is contained in the left hand side, because 
the equality holds for $B^p$.  
\end{proof}

\begin{defn}  \rm

Since $\Pi^{(n)}(\K(X^{\otimes n}))({\rm P})$ and 
  $\left(\bigoplus_{1 \le p \le n,\, {\rm R} \in
 B_{\gamma},\,h^p({\rm R})={\rm P}}\tilde{\rm C}({\rm P},{\rm R},p) \right)$ 
are not orthogonal, we shall modify $\tilde{\rm C}({\rm P},{\rm R},p)$ 
to make them orthogonal.  
We define the modification ${{\rm C}}({\rm P},{\rm R},p) $ by 
\[
{{\rm C}}({\rm P},{\rm R},p) 
:= (I-Q_{\gamma^n}({\rm P}))\tilde{\rm C}({\rm P},{\rm R},p) \not= 0, 
\]
where $Q_{\gamma^n}({\rm P}) \in M_{N^n}(\comp)$ is the projection 
onto the subspace $Z_{\gamma^n}({\rm P}) \subset {\mathbb C}^{N^n}$. 
Since the subspace $Z_{\gamma^n}({\rm P})$ is invariant 
under any operator in $\tilde{\rm C}({\rm P},{\rm R},p)$, 
the projection $Q_{\gamma^n}({\rm P})$ commutes with any operator in 
$\tilde{\rm C}({\rm P},{\rm R},p)$. Therefore 
${{\rm C}}({\rm P},{\rm R},p) $ is a $C^*$-algebra and 
is isomorphic to $M_{N^q}(\comp)$. 
We may write it as follows:
\begin{align*}
    & {{\rm C}}({\rm P},{\rm R},p) 
 =  \{\,
\sum_{0 \le l \le e_{\rm R}-1} 
\pi_{q,i_1\dots i_{p-1}j({\rm R},l)|_{(N,p)}}
(A) \\
 &  - \frac{1}{e_{\rm R}}\sum_{0 \le m,n \le e_{\rm R}-1} \sigma_{q,i_1\dots
 i_{p-1}j({\rm R},m)|_{(N,p)},i_1\dots i_{p-1}j({\rm R},n)|_{(N,p)}(A) }
\,|\, A \in M_{N^{q}}(\comp)
\,\}.  
\end{align*}
\end{defn}

\begin{thm}Let ${\rm P} \in P_{\gamma}$. Then the fiber of $
\Pi^{(n)}(\F^{(n)}) \subset C([0,1], M_{2^n}({\mathbb C}))$ 
at ${\rm P}$ is given by 
\begin{align*}
\Pi^{(n)}(\F^{(n)})({\rm P})
   = &  \Pi^{(n)}(\K(X^{\otimes n}))({\rm P})
   \oplus \left(\bigoplus_{1 \le p \le n,\, {\rm R} \in
 B_{\gamma},\,h^p({\rm R})={\rm P}}{\rm C}({\rm P},{\rm R},p) \right),    
\end{align*}
and the right hand side is a direct sum as *-algebras. 
\end{thm}
\begin{proof} 
For any $T \in {\tilde{\rm C}}({\rm P},{\rm R},p)$,
$$
T = Q_{\gamma^n}({\rm P})T  + (I-Q_{\gamma^n}({\rm P}))T, 
$$ 
and $Q_{\gamma^n}({\rm P})T = Q_{\gamma^n}({\rm P})TQ_{\gamma^n}({\rm P})$  
is in  $\Pi^{(n)}(\K(X^{\otimes n}))({\rm P})$. 
These imply the conclusion. 
\end{proof}

The following Proposition will be used later to describe the model 
traces in terms of matrix representations. 

\begin{prop}\label{prop:orthogonal} In the situation of the above Theorem,
for $q = n-p$ with $q+1 \leq s \leq n$, 
${\rm C}({\rm P},{\rm R},p)$ is orthogonal to 
$\Pi^{(n)}(\K(X^{\otimes s})\otimes I)({\rm P})$. 
\end{prop}
\begin{proof}
Put $Q = \gamma_{i_{n-s}} \circ
 \cdots \circ \gamma_{i_1}({\rm P})$. Consider 
${\rm C}({\rm Q},{\rm R},s-q)$ in 
$\Pi^{(s)}(\F^{(s)})({\rm Q})$. Then we see that 
${\rm C}({\rm Q},{\rm R},s-q)$ is orthogonal  to
$\Pi^{(s)}(\K(X^{\otimes s})({\rm Q})$.

Since there exists no branched points,
$$
{\rm C}({\rm P},{\rm R},p) = 
\pi_{s,i_1\dots i_{n-s}|_{(N,n-s)}}({\rm C}({\rm Q},{\rm R},s-q)).
$$
Keep the $i_1\dots i_{n-s}|_{(N,n-s)}$-th block-diagonal component of 
$\Pi^{(n)}(\K(X^{\otimes s})\otimes I)({\rm P})$ and 
make all the other components zero. Then we get the matrix 
$$
\pi_{s,i_1\dots i_{n-s}|_{(N,n-s)}}\Pi^{s}(\K(X^{\otimes s})(Q). 
$$
Therefore ${\rm C}({\rm P},{\rm R},p)$ is orthogonal to 
$\pi_{s,i_1\dots i_{n-s}|_{(N,n-s)}}\Pi^{s}(\K(X^{\otimes s})(Q)$.
Moreover it is clear that ${\rm C}({\rm P},{\rm R},p)$ is orthogonal
to any other block-diagonal component of 
$\Pi^{(n)}(\K(X^{\otimes s})\otimes I)({\rm P})$. 
Therefore
${\rm C}({\rm P},{\rm R},p)$ is orthogonal to 
$\Pi^{(n)}(\K(X^{\otimes s})\otimes I)({\rm P})$.
\end{proof}

\subsection
{Tent map}
We describe the matrix representation for the case of tent map.  
Since proofs are given in general situation, we only present the
results.  

{\rm (Case of the fiber at $x=1$ )}: \\
  We fix $n$ and consider the fiber $\Pi^{(n)}(\F^{(n)})(1)$ of 
$\Pi^{(n)}(\F^{(n)})$ at $x=1$.  
We investigate $p-q$ orbit through the branched point $\frac{1}{2}$. 

\[
{\rm 1} 
\overset{\gamma_0,\gamma_1}
\Longrightarrow {\rm \frac{1}{2}}
\begin{matrix}
\ & \nearrow \\
\ & ...      \\
\ & \searrow
\end{matrix}
\begin{matrix}
\ & \nearrow \\
\ & ...      \\
\ & \searrow
\end{matrix}
\]

The following is the condition that a matrix is contained in the
image of a compact algebra.  
\begin{align}
& \Pi^{(n)}(\K(X^{\otimes n}))(1)   \nonumber \\
    = &\{T \in M_{2^n}(\comp)\,|\,T_{0i_2i_3\dots i_n,j_1j_2\dots j_n}
=T_{1i_2i_3\dots i_n,j_1j_2\dots j_n}, \quad
  \text{ for each } (i_2,i_3,\dots, i_n) \text{ and }
   (j_1,j_2,j_3,\dots, j_n) \nonumber \\
  & \qquad \qquad \qquad \,\, T_{i_1i_2i_3\dots i_n,0j_2\dots j_n}
=T_{i_1i_2i_3\dots i_n,1j_2\dots j_n} 
  \qquad \text{ for each } (i_1,i_2,i_3,\dots, i_n) \text{ and }
   (j_2,j_3,\dots, j_n) \,\,
 \}  \nonumber \\
= & \left\{ \begin{pmatrix} B & B \\  
\label{eq:eq9}       B & B
     \end{pmatrix}
 \,|\, B \in M_{2^{n-1}}(\comp)
\right\}.  
\end{align}
Since $\gamma_{0}(1)=\gamma_{1}(1)=1/2 \in B_{\gamma}$, only 
$\tilde{\rm C}(1,1/2,1)$ appears for $x=1$.  
Put 
\[
 \tilde{\rm C}(1,1/2,1) = \{\,\pi_{n-1,0|_{(2,1)},1|_{(2,1)}}(A)\,|\,A
 \in 
 M_{2^{n-1}}(\comp)\,\} = \left\{ \begin{pmatrix} A & O \\ O & A
	 	\end{pmatrix} \,|\,
    A \in M_{2^{n-1}}(\comp) \right\}, 
\]
and put 
\begin{align*}
 {\rm C}(1,1/2,1) & =  \{\,\pi_{n-1,0|_{(2,1)},1|_{(2,1)}}(A)
- \frac{1}{2}\sigma_{n-1,0|_{(2,1)},1|_{(2,1)}}(A)
\,|\,A
 \in 
 M_{2^{n-1}}(\comp)\,\} \\
& = 
\left\{ \begin{pmatrix} (1/2)A & -(1/2)A \\ -(1/2)A & (1/2)A
	 	\end{pmatrix} \,|\,
    A \in M_{2^{n-1}}(\comp) \right\}.  
\end{align*}
Then we get the following matrix representation at $x=1$: 
\begin{align*}
 &  \Pi^{(n)}(\F^{(n)}(1)) \\
  = &\Pi^{(n)}(\K(X^{\otimes n}))(1) \oplus C(1,1/2,1) \\
  = &\left\{ \begin{pmatrix} (1/2)A & (1/2)A \\
	      (1/2)A & (1/2)A
	    \end{pmatrix} \,|\, A \in M_{2^{n-1}}(\comp)\right\} 
    \oplus 
\left\{ \begin{pmatrix} (1/2)A & -(1/2)A \\
	      -(1/2)A & (1/2)A
	    \end{pmatrix} \,|\, A \in M_{2^{n-1}}(\comp)\right\} 
\end{align*}

{\rm (Case of the fiber at $x=0$ )}.\\
We investigate $p-q$ orbit through the branched point $\frac{1}{2}$. 

\[
{\rm 0} \overset{\gamma_0}\to {\rm 0} 
\overset{\gamma_0}\to {\rm 0} \dots {\rm 0}
\overset{\gamma_{1}}\to {\rm 1}  
\overset{\gamma_0,\gamma_1}
\Longrightarrow {\rm \frac{1}{2}}
\begin{matrix}
\ & \nearrow \\
\ & ...      \\
\ & \searrow
\end{matrix}
\begin{matrix}
\ & \nearrow \\
\ & ...      \\
\ & \searrow
\end{matrix}
\]

We consider $\Pi^{(n)}(\K(X^{\otimes n}))(0)$.  
The point $0$ is contained in $P_{\gamma}$ and is not contained in
$C_{\gamma}$.

The range of the compact algebra $\Pi^{(n)}(\K(X^{\otimes
n}))(0)$ is
the set of matrices in $M_{2^n}(\comp)$ which satisfy the following 
conditions for rows and columns:
\begin{align*}
& T_{0\cdots 010,j_1\cdots j_n}=  T_{0\cdots 011,j_1\cdots j_n}, \quad 
 T_{i_1\cdots i_n, 0\cdots 010,}=  T_{i_1\cdots i_n, 0\cdots 011},   \\
& T_{0\cdots 010i_n,j_1\cdots j_n}=  T_{0\cdots 011i_n,j_1\cdots j_n}, \quad 
 T_{i_1\cdots i_n, 0\cdots 010j_n,}=  T_{i_1\cdots i_n, 0\cdots
 011j_n},   \\
 & \dots \\
& T_{0\cdots 010i_{k}\cdots i_n,j_1\cdots j_n}=  T_{0\cdots 011i_k\cdots
 i_n,j_1\cdots j_n}, \quad 
 T_{i_1\cdots i_n, 0\cdots 010j_{k}\cdots j_n,}=  T_{i_1\cdots i_n, 0\cdots
 011j_k\cdots j_n},   \\
 & \dots \\
& T_{010i_{3}\cdots i_n,j_1\cdots j_n}=  T_{011i_3\cdots
 i_n,j_1\cdots j_n}, \quad 
 T_{i_1\cdots i_n, 010j_{3}\cdots j_n,}=  T_{i_1\cdots i_n, 011j_3\cdots
 j_n},  \\
 &T_{10i_{2}\cdots i_n,j_1\cdots j_n}=  T_{11i_2\cdots
 i_n,j_1\cdots j_n}, \quad 
 T_{i_1\cdots i_n, 10j_{2}\cdots j_n,}=  T_{i_1\cdots i_n, 11j_2\cdots
 j_n}.  
\end{align*}

We note that for $2 \le p \le n$ 
$\gamma_{i_{p-1}}\circ \gamma_{i_{p-2}} \circ \dots \circ
\gamma_{i_1}(0) \in C_{\gamma}$ if and only if 
$(i_1,i_2,\dots,i_{p-2},i_{p-1})=(0,0,\dots,0,1)$.  

For each $2 \le p \le n$, we define
\[
 \tilde{\rm C}(0,1/2,p) = \{\, \pi_{n-p,0\dots 010|_{(2,p)},0\dots
 011|_{(2,p)}}(A)\,|\,
A \in M_{2^{n-p}}(\comp)\,\}, 
\]
and also define 
\begin{align*}
{\rm C}(0,1/2,p)= & \{ \pi_{n-p,0\dots 010|_{(2,p)},0\dots
 011|_{(2,p)}}(A) - \sigma_{n-p,0\dots 010|_{(2,p)},0\dots
 011|_{(2,p)}}((1/2)A)\,|\, A \in M_{2^{n-p}}(\comp)\,\} \\
 = & 
\left\{\,\begin{pmatrix} O & O & O & O\\
                        O & (1/2)A & -(1/2)A & O \\
                        O & -(1/2)A & (1/2)A & O \\
                        O & O & O & O
	\end{pmatrix}\,|\, A \in M_{2^{n-p}}(\comp)
\,\right\}.  
\end{align*}

{\rm (Matrix representation)}: \\

The expression of the total algebra is as follows:
\begin{align*}
 \Pi^{(n)}(\F^{(n)})(x) & = M_{2^n}(\comp) \qquad 0<x<1,  \\
 \Pi^{(n)}(\F^{(n)})(1) & = \Pi^{(n)}(\K(X^{\otimes n}))(1) \oplus 
                   {\rm C}(1,1/2,1) 
\simeq M_{2^{n-1}}(\comp)\oplus 
                      M_{2^{n-1}}(\comp),(n \geq 1) 
\\ 
 \Pi^{(n)}(\F^{(n)})(0) & =  \Pi^{(n)}(\K(X^{\otimes n}))(0) \oplus 
                   \bigoplus_{p=2}^{n} {\rm C}(0,1/2,p)
                 \simeq M_{2^{n-1}+1}(\comp) \oplus
 \bigoplus_{p=0}^{n-2}M_{2^{p}}(\comp).  
\end{align*}
Moreover we have the following:
\begin{align*}
\Pi^{(0)}(\F^{(0)})(x) & 
= \Pi^{(0)}(A)(x)\simeq \comp \qquad 0 \leq x \leq1,  \\
\Pi^{(1)}(\F^{(1)})(x) & \simeq M_{2}(\comp) \qquad 0\leq x<1,  \\
\Pi^{(1)}(\F^{(1)})(1) & \simeq \comp \oplus \comp,  \\
\Pi^{(2)}(\F^{(2)})(0) & \simeq M_{3}(\comp) \oplus \comp        \\
\Pi^{(2)}(\F^{(2)})(1) & \simeq M_{2}(\comp) \oplus M_{2}(\comp) \\
\Pi^{(3)}(\F^{(3)})(0) & \simeq M_{5}(\comp) \oplus \comp \oplus M_{2}(\comp) \\
\Pi^{(3)}(\F^{(3)})(1) & \simeq M_{4}(\comp) \oplus M_{4}(\comp) 
\end{align*}

\subsection
{Sierpinski Gasket case}
Recall the set $B_{\gamma}=\{\,{\rm S},{\rm T},{\rm U}\,\}$ of 
branched points and the set 
$C_{\gamma}=\{\,{\rm P},{\rm Q},{\rm R}\,\}$ of branch values. 
We list the image of branched points and branch values under
the contractions $(\gamma_0,\gamma_1,\gamma_2)$ for the convenience 
of computation. .  
\begin{align*}
 & \gamma_0({\rm P}) ={\rm P}, \quad \gamma_1({\rm P}) = {\rm T}, \quad \gamma_2({\rm P}) = {\rm T}, \\
 & \gamma_0({\rm Q}) ={\rm S}, \quad \gamma_1({\rm Q}) = {\rm S}, \quad \gamma_2({\rm Q}) = {\rm {\rm R}}, \\
 & \gamma_0({\rm R}) ={\rm U}, \quad \gamma_1({\rm R}) = {\rm Q}, \quad
 \gamma_2({\rm R}) = {\rm U}.  
\end{align*}

{\rm (Picture of the fiber at $P$ )}.\\
We investigate $p-q$ orbits through the branched point $T$. 

\[
{\rm P}\overset{\gamma_0}\to {\rm P}
\overset{\gamma_0}\to {\rm P} \dots 
\overset{\gamma_0}\to {\rm P}  
\overset{\gamma_1, \gamma_2}
\Longrightarrow {\rm T} 
\begin{matrix}
\ & \nearrow \\
\ & ...      \\
\ & \searrow
\end{matrix}
\begin{matrix}
\ & \nearrow \\
\ & ...      \\
\ & \searrow
\end{matrix}
\]

The beginning $\gamma_{0}$ may be omitted. \\

{\rm (Picture of the fiber at $Q$ )}.\\
We investigate $p-q$ orbits through the branched point $S$. 

\[
{\rm Q} \overset{\gamma_2}\to {\rm R} 
\overset{\gamma_1}\to {\rm Q} 
\overset{\gamma_2}\to {\rm R}
\overset{\gamma_1}\to {\rm Q}
\dots 
\overset{\gamma_2}\to {\rm R} 
\overset{\gamma_1}\to {\rm Q}
\overset{\gamma_0, \gamma_1}
\Longrightarrow {\rm S} 
\begin{matrix}
\ & \nearrow \\
\ & ...      \\
\ & \searrow
\end{matrix}
\begin{matrix}
\ & \nearrow \\
\ & ...      \\
\ & \searrow
\end{matrix}
\]

We investigate $p-q$ orbits through the branched point $U$. 

\[
{\rm Q} \overset{\gamma_2}\to {\rm R} 
\overset{\gamma_1}\to {\rm Q} 
\overset{\gamma_2}\to {\rm R}
\overset{\gamma_1}\to {\rm Q}
\dots 
\overset{\gamma_1}\to {\rm R}
\overset{\gamma_0, \gamma_2}
\Longrightarrow {\rm U} 
\begin{matrix}
\ & \nearrow \\
\ & ...      \\
\ & \searrow
\end{matrix}
\begin{matrix}
\ & \nearrow \\
\ & ...      \\
\ & \searrow
\end{matrix}
\]

The picture of the fiber at ${\rm R}$ is similar with 
the picture of the fiber at ${\rm Q}$. 

Let $1 \le p \le n$.  
Elements in $\Pi^{(n)}(\K(X^{\otimes n}))({\rm P})$ are described as
elements in $M_{3^n}(\comp)$ satisfying the following all conditions:
\begin{itemize}
 \item $(1,i_2,\dots,i_n)$th row and $(2,i_2,\dots,i_n)$th row are equal,
       and $(1,i_2,\dots,i_n)$th column and $(2,i_2,\dots,i_n)$th column
       are equal.  
 \item $(0,1,i_3,\dots,i_n)$th row and $(0,2,i_3,\dots,i_n)$th row are
       equal and $(0,1,i_3,\dots,i_n)$th column and
       $(0,2,i_3,\dots,i_n)$th column are equal.  
 \item $(0,\dots,0,1,i_{k+1},\dots,i_n)$th row and 
       $(0,\dots,0,2,i_{k+1},\dots,i_n)$th row are equal and \\
       $(0,\dots,0,1,i_{k+1},\dots,i_n)$th column and 
       $(0,\dots,0,2,i_{k+1},\dots,i_n)$th column are equal for $(1 \le
       k \le n-1)$.  
 \item $(0,\dots,0,1)$th row and $(0,\dots,0,2)$th row are equal and 
       $(0,\dots,0,1)$th column and $(0,\dots,0,2)$th column are equal.  
\end{itemize}
The condition of compactness at ${\rm Q}$ and ${\rm R}$ are similar, and the
dimension of the image under matrix representation of the compact
algebras are equal for each branch value.  We note that 
$\gamma_{i_{p-1}} \circ \gamma_{i_{p-2}}\circ \dots \circ
\gamma_{i_1}({\rm P}) \in C_{\gamma}$ if and only if
$(i_1,\dots,i_{p-2},i_{p-1}) = (0,\dots,0,0)$.  
Put 
\[
 \tilde{\rm C}({\rm P},{\rm T},p)= \{\,\pi_{n-p,00\dots
 1|_{(3,p)},00\dots 2|_{(3,p)}}(A)
          \,|\, A \in M_{3^{n-p}}(\comp)\,\}.    
\]
We also put 
\[
 {\rm C}({\rm P},{\rm T},p) =  \left\{\, \pi_{n-p,00\dots
 1|_{(3,p)},00\dots
 2|_{(3,p)}}(A) - (1/2)\sigma_{n-p,00\dots 1|_{(3,p)}, 00\dots 2|_{(3,p)}}
 (A)\,|\, A \in M_{3^{n-p}}(\comp)\,
\right\}.    
\]
Then by the definition of ${\rm C}({\rm P},{\rm T},p)$ it holds
\[
 \Pi^{(n)}(\K(X^{\otimes n}))({\rm P}) + \tilde{\rm C}({\rm P},{\rm
 T},p)
   =  \Pi^{(n)}(\K(X^{\otimes n}))({\rm P}) \oplus {\rm C}({\rm P},{\rm T},p).  
\]
Then it holds
\[
  \Pi^{(n)}(\F^{(n)})({\rm P}) = \Pi^{(n)}(\K(X^{\otimes n}))({\rm P})
  \oplus  \bigoplus_{p=0}^{n-1}{\rm C}({\rm P},{\rm T},p)
\]

Next we consider the cases Q and R.  
Since Q and R are symmetric, we only consider the 
case Q.  Let $1 \le p \le n-1$.  Then $\gamma_{i_{p-1}}\circ \dots \circ 
\gamma_{i_1}({\rm Q}) \in C_{\gamma}$ if and only if $(i_1,\dots,i_{p-1}) =
(2,1,2,1,\dots)$
We note that the form changes according to the the parity of $p$.  
For the case $p$ is even, put 
\[
  \tilde{\rm C}({\rm Q},{\rm U},p) 
   = \{\,\pi_{n-p,2121\dots 20|_{(3,p)},2121\dots 22|_{(3,p)}}
    (A) \,|\, A \in M_{3^{n-p}}(\comp)\,
 \}.  
\]
For the case $p$ is odd, put
\[
  \tilde{\rm C}({\rm Q},{\rm S},p) 
   = \{\,\pi_{n-p,2121\dots 10|_{(3,p)},2121\dots 11|_{(3,p)}}
    (A) \,|\, A \in M_{3^{n-p}}(\comp)\,
 \}.  
\]
Moreover, for $p$ even, put
\[
  {\rm C}({\rm Q},{\rm U},p) 
   = \{\,\pi_{n-p,2121\dots 20|_{(3,p)},2121\dots 22|_{(3,p)}}
    (A) - \sigma_{n-p,2121\dots 20|_{(3,p)},2121\dots 22|_{(3,p)}}
    ((1/2)A) \,|\, A \in M_{3^{n-p}}(\comp)\,
 \}, 
\]
and for $p$ odd, put
\[
  {\rm C}({\rm Q},{\rm S},p) 
   = \{\,\pi_{n-p,2121\dots 20|_{(3,p)},2121\dots 22|_{(3,p)}}
    (A) - \sigma_{n-p,2121\dots 10|_{(3,p)},2121\dots 11|_{(3,p)}}
    ((1/2)A) \,|\, A \in M_{3^{n-p}}(\comp)\,
 \}.  
\]

We put 
\[
 {\rm H}^p = 
\begin{cases}
 & {\rm C}({\rm Q},{\rm U},p) \quad (p \text{ is even }),  \\
 & {\rm C}({\rm Q},{\rm S},p) \quad (p \text{ is odd }).  
\end{cases}
\]
Then also by Theorem \ref{th:representation}, 
\[
 \Pi^{(n)}(\F^{(n)})({\rm Q})  
   =  \Pi^{(n)}(\K(X^{\otimes n}))({\rm Q}) \oplus 
        \bigoplus_{p=1}^{n} {\rm H}^{p}.  
\]
We consider the case of R.  For $p$ even, put 
\[
  {\rm C}({\rm R},{\rm S},p) 
   = \{\,\pi_{n-p,1212\dots 10|_{(3,p)},1212\dots 11|_{(3,p)}}
    (A)
 - \sigma_{n-p,1212\dots 10|_{(3,p)},1212\dots 11|_{(3,p)}}
    ((1/2)A) \,|\, A \in M_{3^{n-p}}(\comp)\,
 \}, 
\]
and for $p$ odd, put
\[
  {\rm C}({\rm R},{\rm U},p) 
   = \{\,\pi_{n-p,1212\dots 20|_{(3,p)},1212\dots 22|_{(3,p)}}
    (A) - 
\sigma_{n-p,1212\dots 20|_{(3,p)},1212\dots 22|_{(3,p)}}
    ((1/2)A) \,|\, A \in M_{3^{n-p}}(\comp)\,
 \}.  
\]

Similarly we put 
\[
 {\rm I}^p = 
\begin{cases}
 & {\rm C}({\rm R},{\rm S},p) \quad (p \text{ is even }),  \\
 & {\rm C}({\rm R},{\rm U},p) \quad (p \text{ is odd }).   
\end{cases}
\]
Then by Theorem \ref{th:representation}, it holds that
\[
 \Pi^{(n)}(\F^{(n)})({\rm R})
   =  \Pi^{(n)}(\K(X^{\otimes n}))({\rm R}) \oplus
        \bigoplus_{p=1}^{n} {\rm I}^p.  
\]

Therefore we get the matrix representation of $\F^{(n)}$ as follows:
\begin{align*}
 \Pi^{(n)}(\F^{(n)})({\rm X}) \simeq &M_{3^n}(\comp),  \quad {\rm X} \ne
 {\rm P}, \, {\rm Q}, \, {\rm R} \\
  \Pi^{(n)}(\F^{(n)})({\rm P}) = &\Pi^{(n)}(\K(X^{\otimes n}))({\rm P})
  \oplus  \bigoplus_{p=1}^{n}{\rm C}({\rm P},{\rm T},p)
 \simeq M_{(1/2)(3^{n}+1)}(\comp) \oplus
 \bigoplus_{p=1}^{n}M_{3^{p-1}}(\comp), 
 \\
 \Pi^{(n)}(\F^{(n)})({\rm Q})=& \Pi^{(n)}(\K(X^{\otimes n}))({\rm Q}) \oplus
        \bigoplus_{p=1}^{n}{\rm H}^{n-p+1}\simeq M_{(1/2)(3^{n}+1)}(\comp)
 \oplus \bigoplus_{p=1}^{n}M_{3^{p-1}}(\comp), 
\\
 \Pi^{(n)}(\F^{(n)})({\rm R})=&\Pi^{(n)}(\K(X^{\otimes n}))({\rm R}) \oplus
        \bigoplus_{p=1}^{n}{\rm I}^{n-p+1} \simeq
 M_{(1/2)(3^{n}+1)}(\comp) \oplus
 \bigoplus_{p=1}^{n}M_{3^{p-1}}(\comp).  
\end{align*}
Three algebras are isomorphic as \cst -algebras.  But the realizations
of them in $M_{3^n}(\comp)$ are different.

\section
{Calculation of K-groups of the cores}
\subsection
{Tent map case}
From the explicit calculation of matrix representation, we have  that
\begin{align*}
 \Pi^{(n)}(\F^{(n)})(x) & = M_{2^n}(\comp), \qquad 0<x<1,  \\
 D(1):= & \Pi^{(n)}(\F^{(n)})(1) = \Pi^{(n)}(\K(X^{\otimes n}))(1) \oplus 
 {\rm C}(1,1/2,1) 
 \simeq   M_{2^{n-1}}(\comp)\oplus M_{2^{n-1}}(\comp). \\
 D(0):= \Pi^{(n)}(\F^{(n)})(0) & =  \Pi^{(n)}(\K(X^{\otimes n}))(0) \oplus 
                   \bigoplus_{p=0}^{n-2}{\rm C}(0,1/2,p)
                 \simeq M_{2^{n-1}+1}(\comp) \oplus
 \bigoplus_{p=0}^{n-2}M_{2^{p}}(\comp). 
\end{align*}

We introduce the following notation:
\begin{align*} 
  J = & \{\,T \in {\rm C}([0,1],M_{2^n}(\comp))\,|\,T(0)=O,\,
 T(1)=O\,\},    \\
  B = & \{\,T \in {\rm C}([0,1],M_{2^n}(\comp))\,|\,T(0) \in D(0),\, T(1)
 \in D(1)\,\}= \Pi^{(n)}(\F^{(n)}),    \\
  C = & D(0) \oplus D(1).  
\end{align*}
Then we have the following exact sequence: 
\[
 \{0\}  \to   J  \to B  \to C  \to \{0\}.  
\]
From this, we have the following 6-term exact sequence of K-groups:  
\[
\begin{CD}
   {\rm K}_0(J) @>{}>> {\rm K}_0(\F^{(n)}) @>{}>> {\rm K}_0(C) \\
   @A{{\rm ind}}AA
    @.
     @VV{{\rm exp}}V \\
   {\rm K}_1(C) @<<{}< {\rm K}_1(\F^{(n)}) @<<{}< {\rm K}_1(J). 
\end{CD}  
\]
Substituting known K-groups, we have that 
\[
\begin{CD}
   \{\mathbf 0\} @>{}>> {\rm K}_0(\F^{(n)}) @>{}>> {\mathbb Z}^{2}\oplus
 {\mathbb Z}^{n}  \\
   @A{{\rm ind}}AA
    @.
     @VV{{\rm exp}}V \\
   \{\mathbf 0\}  @<<{}< {\rm K}_1(\F^{(n)}) @<<{}< {\rm K}_1(J). 
\end{CD}  
\]
We calculate the exponential map (exp map) from ${\rm K}_0(C)$ to ${\rm
K}_1(J)$.  
Use 12.2 in  Rodam {\it et al.} \cite{Ro}, for example.  
Let $p_1$ be a minimal projection in 
$\Pi^{(n)}(\K(X^{\otimes n}))(1) 
                    \simeq M_{2^{n-1}}(\comp)$.
There exists an Hermite
element $T \in M_{2^n}(C[0,1])$ with $T(1)=p_1$, $T(0)=0$.  
Then $\delta(([p_1],0)) =-[{\rm exp}(2\pi i T)]=-1$.  
Thus  $-1$ is the image of $[p_1]$ in ${\rm K}_1(J)$ under the exponential
map.  On the other hand, let $p_2$ be a minimal projection in 
${\rm C}(1,1/2,1)  \simeq M_{2^{n-1}}(\comp)$. Then 
There exists an Hermite element $T \in M_{2^n}(C[0,1])$ with $T(1)=0$,
$T(0)=p_2$.  
Thus $\delta((0,[p_2]))=1 \in {\rm K}_1(J)$.  
For a positive element in ${\mathbb Z}^2 \oplus {\mathbb Z}^n$, 
exp map is described as
\[
 {\rm exp}((m_1,m_2,r_1,\dots,r_n)) = -m_1-m_2 + r_1 + \cdots +r_n.  
\]
By the extension of exp map by homomorphism property, it is shown that ${\rm
exp}$ is given on the whole part of ${\mathbb Z}^2 \oplus {\mathbb
Z}^n$.  

${\rm K}_1(J)\simeq {\mathbb Z}$ and the exp map is surjective.  
Then for each $n$, it holds that ${\rm K}_1(\F^{(n)})=\{\mathbf 0\}$.  
On the other hand, it holds that 
\[
{\rm K}_0(\F^{(n)}) \simeq {\rm ker}({\rm exp})
   \simeq 
\{\,((m_1,m_2),(r_1,\dots,r_n)) \in {\mathbb Z}^{2} \oplus {\mathbb Z}^n\,|\,
        m_1 + m_2 = r_1 + \cdots +r_n \,\} \simeq {\mathbb Z}^{n+1}.  
\]

Next, we calculate the embeddings of $\{\,{\rm
K}_0(\F^{(n)})\,\}_{n=0,1,2,\dots}$.  Let $p^1_1$ be a minimal
projection of ${\rm C}(1,1/2,1)$
and $p^1_2$ be a minimal projection of $\Pi^{(n)}(\K(X))(1)$.  
Then it holds that 
\[
 {\rm K}_0(D(1)) = \{\,m_1[p^1_1]+m_2[p^1_2]\,|\,m_1,m_2 \in {\mathbb
 Z}\,\} \simeq  \{\,(m_1,m_2)\,|\,m_1,m_2 \in {\mathbb
 Z}\,\}.  
\]
Let $p_i^0$ be a minimal projection of ${\rm C}(0,1/2,i+1)$ for $1 \le i
\le n-1$ and $p_n^0$ be a minimal projection of 
$\Pi^{(n)}(\K(X^{\otimes n
}))(0)$.  Then it holds 
\[
  {\rm K}_0(D(0)) = \{\,r_1[p^0_1]+\dots + r_n[p^0_n]\,|\,r_i \in {\mathbb
  Z}\} \simeq  \{\,(r_1,\dots,r_n)\,|\,r_i \in {\mathbb  Z}\}.  
\]
Let $1\le i \le n$, and we calculate the embedding of 
$((1,0),(0,\dots,1_{i},\dots,0)) \in {\mathbb Z}^2 \oplus {\mathbb
Z}^n$.  We write as $\Psi^{n,n+1}$ the embedding map from ${\rm
K}_0(\F^{(n)})$ to ${\rm K}_0(\F^{(n+1)})$. 

We denote by 
$r^n_1$ a minimal projection of ${\rm C}(1,1/2,1)$, and by $q^n_i$ a
minimal projection of ${\rm C}(0,1/2,n-i+1)$.  We take $T\in \F^{(n)}$ such
that $T(0)=q^n_i$, $T(1)=r^n_1$.  We calculate the forms of
$\Pi^{(n+1)}(T)(0)$ and $\Pi^{(n+1)}(T)(1)$.  
It holds that 
\[
\Pi^{(n+1)}(T)(0) = \diag( \Pi^{(n)}(T)(0),\Pi^{(n)}(T)(1)),  \quad 
\Pi^{(n+1)}(T)(1) =  \diag( \Pi^{(n)}(T)(1/2),\Pi^{(n)}(T)(1/2)).  
\]
Then it holds that 
\[
 \Psi^{(n,n+1)}((1,0),(0,\dots,1_{i},\dots,0))
   = ((1,1),(1_{1},\dots,1_{i+1},\dots,0)).  
\]
Similarly, it holds that
\begin{align*}
 \Psi^{(n,n+1)}((0,1),(0,\dots,1_{i},\dots,0))
   = &((1,1),(0,\dots,1_{i+1},\dots,1_{n+1})) \\
 \Psi^{(n,n+1)}((1,0),(0,\dots,0,\dots,1_n))
   = &((1,1),(1_{1},\dots,0,\dots,1_{n+1})) \\
 \Psi^{(n,n+1)}((0,1),(0,\dots,0,\dots,1_n))
   = &((1,1),(0,\dots,0,\dots,2_{n+1})).  
\end{align*}
The algebra ${\rm K}_0(\F^{(\infty)})$ can be described as the inductive
limit algebra under the embedding given by $\Psi^{n,n+1}$.  

As an basis in ${\rm K}_0(\F^{(n)})$, we can take
\[
  e_i^n =((1,0),(0,\dots,0,1_i,0,\dots,0)) \quad (i=1,\dots,n), \quad 
  e_{n+1}^n =((0,1),(1,0,\dots,0,0)).  
\]
We note that 
\[
  ((0,1),(0,\dots,0,1_i,0,\dots,0)) = e^n_{n+1} + e^n_i -e^n_1.  
\]
Then the action of $\Psi^{(n,n+1)}$ on the bases is expressed as
\begin{align*}
  \Psi^{(n,n+1)}(e^n_i) = & ((1,1),(1_1,0,\dots,0,1_{i+1},0,\dots,0))
                        =  e^{n+1}_{i+1} + e^{n+1}_{n+2} \quad
 (i=1,\dots,n), \\
  \Psi^{(n,n+1)}(e^n_{n+1}) = & 
      ((1,1),(0,\dots,2_{n+1})) \\
       = &((1,0),(0,\dots,0,1_{n+1})) + ((0,1),(1_1,\dots,0))
 -((1,0),(1_1,\dots,0))  \\ 
         & + ((1,0),(0,\dots,1_{n+1}))
       = - e^{n+1}_1+  2e^{n+1}_{n+1} + e^{n+1}_{n+2}.  
\end{align*}
The matrix representation $A^{(n,n+1)}$ of $\Psi^{n,n+1}$using the bases
is expressed as
\[
 A^{(n,n+1)}
 = \begin{pmatrix}
     0 & 0 & 0 & \cdots & 0 & -1  \\
     1 & 0 & 0 & \cdots & 0 & 0   \\
     0 & 1 & 0 & \cdots & 0 & 0  \\
     \vdots & \vdots & \vdots & \ddots & \vdots & \vdots \\
     0  & 0 &  0 & \cdots &  1 & 2 \\
     1  & 1 & 1 & \cdots & 1  & 1 
   \end{pmatrix}.  
\]
The matrix $A^{(n,n+1)}$ is an $(n+2)\times (n+1)$ matrix whose entries are
integers.  

\begin{prop}\label{prop:K-group}  The ${\rm K}$ group of the core of 
the \cst -algebra
 associated with the tent map is given as follows:
\begin{align*}
  {\rm K}_0(\F^{(\infty)}) = & \lim_{n \to \infty}\left({\mathbb Z}^{n+1}
 \overset{A^{(n,n+1)}}{\to}
 {\mathbb Z}^{n+2} \right) \\
  {\rm K}_1(\F^{(\infty)}) =& \{\,0\,\}.  
\end{align*}
\end{prop}

\begin{lemma} 
The map $\Phi^{(n,n+1)}$ is
injective for each $n$.  
\end{lemma}
\begin{proof}
Since the matrix $A^{(n,n+1)}$ is of rank $n$, $\Phi^{(n,n+1)}$ is
injective. 
\end{proof}

\begin{thm} \label{thm:dimension-group} 
Let $\F^{(\infty)}$ be the core of the \cst -algebra associated
with tent map, then ${\rm K}_0(\F^{(\infty)})$ is isomorphic to 
the countably generated 
free abelian group ${\mathbb Z}^{\infty}\cong {\mathbb Z}[t]$ 
as an abstract abelian group. 
Moreover for the tracial state ${\tau}^{(\infty)}$ on $\F^{(\infty)}$ 
corresponding to the Hutchinson measure, we have that 
${\tau}^{{(\infty)}*}({\rm K}_0(\F^{(\infty)})) = {\mathbb Z}[\frac{1}{2}]$. 
 
\end{thm}
\begin{proof}
Since each $n$ $\Phi^{(n,n+1)}$ is injective, the canonical map 
of ${\rm K}_0(\F^{(n)})$ to ${\rm K}_0(\F^{(\infty)})$ 
is injective, for example, see Exercises 6.7 in Rodam {\it et al.} \cite{Ro}. 
By Theorem 7.8. of \cite{KW3} we know that the extreme tracial states on 
the core of the \cst -algebra
associated with the tent map are exactly the model traces 
\[
\{\tau^{(\infty)}\} \bigcup
\{\tau^{(1/2,r)}\,|\,r=0,1,2,\dots \}.
\]
where $1/2$ is the unique branched point of the tent map. 
Put $\tau^{(r)} = \tau^{(1/2,r)}$. Then they are  described 
in Example 7.1. of \cite{KW3} as follows: 
We define discrete measures $\mu_i^{(r)}$on $[0,1]$ for $0 \le i \le r$ by
\[
 \mu_i^{(r)}(f) = \frac{1}{2^{r-i}}\sum_{(j_1,\dots,j_{r-i}) \in
\{\,1,2\,\}^{r-i}}
    f(\gamma_{j_1}\circ \cdots \circ \gamma_{j_{r-i}}(1/2)), 
\ \ \ (f \in C[0,1]).
\]
It is enough to define traces $\tau^{(r)}$ on $\F^{(\infty)}$ for $r \ge 0$ 
by the restriction 
$\tau^{(r)}_i = \tau^{(r)}|_{K(X_{\gamma}^{\otimes i})}$ as 
 \[
  \tau^{(r)}_i =
\begin{cases}
 & \pi_i(\mu^{(r)}_i) \quad (0 \le i \le r) \\
 & 0 \quad \quad (r+1 \le i)
\end{cases}
 \]
for each $i$, where $\pi_i(\mu^{(r)}_i)$ is the trace on 
$K(X_{\gamma}^{\otimes i})$ corresponding to $\mu^{(r)}_i$ 
using Morita equivalence. 
We can also define a trace $\tau^{(\infty)}$ on $\F^{(\infty)}$ by
\[
 \tau^{(\infty)}_i = {\pi}_i(\mu_{\infty}),
\]
for each $i$ where $\mu_{\infty} = \mu_H $ 
is the normalized Lebesgue measure on
$[0,1]$ and coincides with the Hutchinson measure.  
The  set of extreme traces on the core of $C^*$-algebras 
associated with the tent map on $[0,1]$ is exactly 
 $\{\,\tau^{(0)},\tau^{(1)},\dots,\tau^{(\infty)}\,\}$.

Fix an integer $n \geq 0$ for a while. 
We need to describe the restriction 
$\tau^{(r)}|_{\F^{(n)}}$ of these model traces $\tau^{(r)}$ 
on $\F^{(n)}$ in termes of matrix representations. It is 
clear that $\tau^{(r)}|_{K(X_{\gamma}^{\otimes n})} = 0$ 
for $n \geq r+1$ by definition. 
We denote  by $tr$ the normalized trace on the full matrix algebra.

Let $n = 0$. Then $\F^{(0)} = A$ and 
$\tau^{(0)}|_{\F^{(0)}}(a) = a(\frac{1}{2})$ for $a \in A$.

Let $n = 1$. Then $\F^{(1)} = A\otimes I + K(X_{\gamma})$ 
and  $\Pi^{(1)}(\F^{(1)}) \subset M_{2}(A)$. And 
$$
\tau^{(1)}|_{\F^{(1)}}(T) = tr(\Pi^{(1)}(T) (\frac{1}{2}))
$$ 
for 
$T \in \F^{(1)}$. Let $E^{(1)}$ be the projection of the fiber 
$$\Pi^{(1)}(\F^{(1)})(1)
= \Pi^{(1)}(K(X_{\gamma}))(1) \oplus C(1,\frac{1}{2},1)
$$ 
onto 
$C(1,\frac{1}{2},1) = {\mathbb C}$. Then 
$$
\tau^{(0)}|_{\F^{(1)}}(T) = E^{(1)}(\Pi^{(1)}(T)(1)).
$$ In fact, 
if $T = a \otimes I$, then 
$E^{(1)}(\Pi^{(1)}(T)(1))= a(\gamma_0(1))= a(\frac{1}{2})$. If 
$T$ is in $K(X_{\gamma})$, then 
$\tau^{(0)}(T) = 0 = E^{(1)}(\Pi^{(1)}(T)(1))$ by definition.

Let $n \geq 2$ in general. By induction, we can show the following: 
Recall that $\Pi^{(n)}(\F^{(n)}) \subset M_{2^n}(A)$ . Then 
$$
\tau^{(n)}|_{\F^{(n)}}(T) = tr(\Pi^{(n)}(T) (\frac{1}{2}))
$$ for 
$T \in \F^{(n)}$. 
Let $E^{(n-1)}$ be the projection of the fiber 
$$
\Pi^{(n)}(\F^{(n)})(1)
= \Pi^{(n)}(K(X_{\gamma}^{\otimes n}))(1) \oplus C(1,\frac{1}{2},1)
$$ 
onto the direct sum component
$C(1,\frac{1}{2},1) = M_{2^{n-1}}(\mathbb C)$. Then 
$$
\tau^{(n-1)}|_{\F^{(n)}}(T) = tr(E^{(n-1)}(\Pi^{(n)}(T)(1))), 
$$
for $T \in \F^{(n)}$.  
Take $r = n-p$ for some $p$ with $2 \leq p \leq n$. 
Let $E^{(n-p)}$ be the projection of the fiber 
$$
\Pi^{(n)}(\F^{(n)})(0)
= \Pi^{(n)}(K(X_{\gamma}^{\otimes n}))(0) \oplus 
\oplus_{p=2}^n C(0,\frac{1}{2},p)
$$
onto the direct sum component 
$C(0,\frac{1}{2},p) = M_{2^{n-p}}(\mathbb C)$.
Then 
$$
\tau^{(n-p)}|_{\F^{(n)}}(T) = tr(E^{(n-p)}(\Pi^{(n)}(T)(0))).
$$ 
for $T \in \F^{(n)}$. In fact, if $T = S\otimes I$ for some 
$S \in L(X_{\gamma}^{\otimes n-p})\cap \F^{(n-p)}$, then 
$ E^{(n-p)}(\Pi^{(n)}(T)(0))= \Pi^{(n-p)}(S)(\frac{1}{2})$, 
since $p-(n-p)$ orbit through the branched point $\frac{1}{2}$ 
is 
\[
{\rm 0} \overset{\gamma_0}\to {\rm 0} 
\overset{\gamma_0}\to {\rm 0} \dots {\rm 0}
\overset{\gamma_{1}}\to {\rm 1}  
\overset{\gamma_0,\gamma_1}
\Longrightarrow {\rm \frac{1}{2}}
\begin{matrix}
\ & \nearrow \\
\ & ...      \\
\ & \searrow
\end{matrix}
\begin{matrix}
\ & \nearrow \\
\ & ...      \\
\ & \searrow
\end{matrix}
\]
if $T = S\otimes I$ for some 
$S \in K(X_{\gamma}^{\otimes k})\cap \F^{(k)}$ with $k \geq n-p$,  
then 
$ E^{(n-p)}(\Pi^{(n)}(T)(0))= 0$, because 
$\Pi^{(n)}(T)(0)$ and $C(0,\frac{1}{2},p)$ are orthogonal by 
Proposition 
\ref{prop:orthogonal}. 
Therefore 
$$
\tau^{(n-p)}|_{\F^{(n)}}(T) = tr(E^{(n-p)}(\Pi^{(n)}(T)(0))).
$$ 
for $T \in \F^{(n)}$.

We define a map 
$$
\varphi : {\rm K}_0(\F^{(\infty)}) 
\rightarrow \prod _{n =0}^{\infty} {\mathbb R}
$$ 
by $\varphi ([p]) = (\tau^{(n)}(p))_n$ for projections $p$ in 
$\F^{(\infty)}$. 
Define  projections $p^0 = I \in A$ and $p^n \in K(X_{\gamma}^{\otimes n})$ 
for $n = 1,2,3,\dots $ by constant matrix function of $J_n$ 
such that for any $x \in [0,1]$ 
$\Pi^{(n)}(p^n)(x)$ is the rank-one projection $J_n$, where    
any entry of $J_n$ is $\frac{1}{2^n}$. Put 
$$
c_n := \varphi ([p^n]) 
= (0,0,\dots,0,\frac{1}{2^n}, \frac{1}{2^n}, \frac{1}{2^n}, \dots ).
$$
We shall show that
$\varphi ({\rm K}_0(\F^{(\infty)}) )$ is isomorphic to 
the countably generated 
free abelian group ${\mathbb Z}^{\infty}$. In fact it is 
enough to show that  
$\varphi ({\rm K}_0(\F^{(\infty)}) )$ is 
generated by these ${\mathbb Z}$-independent elements 
$\{ c_n \in \prod _{n =0}^{\infty} {\mathbb R}\ 
 | \ n = 0,1,2,\dots \}$. We note that 
$\varphi ({\rm K}_0(\F^{(\infty)}) )$ contains 
$c_n -2c_{n+1} = (0,0,\dots,0, \frac{1}{2^n},0,\dots )$ and 
$$
\tau^{(n)}_*({\rm K}_0(\F^{(\infty)}) ) 
= \{\frac{m}{2^n} \ | \ m \in {\mathbb Z}\}. 
$$
Consider the canonical inclusion $i_n : \F^{(n)} \rightarrow \F^{(\infty)}$, 
then 
$$
\varphi ({\rm K}_0(\F^{(\infty)}) ) 
 = \cup_{n=0}^{\infty}\varphi(i_n*({\rm K}_0(\F^{(n)}))).
$$ 
We shall also show that $\varphi$ is one to one. Therefore 
it is enough to show that 
$\varphi(i_n*({\rm K}_0(\F^{(n)})))$ is contained in the 
subgroup generated by 
$\{ c_m \in \prod _{n =0}^{\infty} {\mathbb R}
 | \  m = 0,1,2, \dots \}$ and 
$\varphi \circ i_n*: {\rm K}_0(\F^{(n)}) \rightarrow 
\prod _{n =0}^{\infty} {\mathbb R}$ is one to one  
for each $n$ by Proposition 6.2.5 in \cite{Ro}. 
Let $\{e_1^n, e_2^n, \dots, e_{n+1}^n\}$ be a basis 
of ${\rm K}_0(\F^{(n)})$ chosen before. Then for $k = 1,2,\dots,n+1$, 
we have 
$$\tau^{(0)}(e_k^n) = \delta_{k,2}, \ \  
\tau^{(1)}(e_k^n) = \frac{1}{2}\delta_{k,3},
$$
$$ 
\tau^{(2)}(e_k^n) = \frac{1}{2^2}\delta_{k,4},\ \ 
\dots, 
\tau^{(n-2)}(e_k^n) = \frac{1}{2^{n-2}}\delta_{k,n}, \ \ 
\tau^{(n-1)}(e_k^n) = \frac{1}{2^{n-1}}\delta_{k,n+1}, 
$$ and 
$$\tau^{(n)}(e_k^n) = \frac{1}{2^{n}}.
$$ Moreover for any 
$m \geq n+1$, 
$$
\tau^{(m)}(e_k^n) = \frac{1}{2^{n}}.
$$ 
Hence the image of the basis 
$\{e_1^n, e_2^n, \dots, e_{n+1}^n\}$ under $\varphi$ 
is independent over ${\mathbb Z}$. Thus 
$\varphi \circ i_n*: {\rm K}_0(\F^{(n)}) \rightarrow 
\prod _{n =0}^{\infty} {\mathbb R}$ is one to one. 
We also have that $\varphi(i_n*({\rm K}_0(\F^{(n)})))$ is 
contained in the subgroup generated by 
generated by 
$\{ \ c_m \in \prod _{n =0}^{\infty} {\mathbb R}
 | \  m = 0,1,2, \dots \}$. 
Therefore   
${\rm K}_0(\F^{(\infty)})$ is isomorphic to  ${\mathbb Z}^{\infty}$. 
It is easy to see that 
${\tau}^{{\infty}*}({\rm K}_0(\F^{\infty}) = {\mathbb Z}[\frac{1}{2}]$.

\end{proof}

\begin{rem}Compare the tent map with the following 
self-similar map.  Let $K=[0,1]$, $\gamma_0(y)=(1/2)y$, $\gamma_1(y) =
 (1/2)y+1/2$.  Then $\gamma = (\gamma_0,\gamma_1)$ has no branched points.  
For this $\gamma$, $\O_{\gamma} \simeq \O_{2}$, and $\F^{(\infty)}$ is the UHF \cst
 -algebra $M_{2^{\infty}}(\comp)$, and it is well known that ${\rm
 K}_0(\F^{(\infty)})$ is the group ${\mathbb Z}[\frac{1}{2}]$ of
 $2$-adic integers and totally ordered. Note that 
${\mathbb Z}[\frac{1}{2}]$ is not isomorphic to ${\mathbb Z}^{\infty}$. 
In ${\mathbb Z}[\frac{1}{2}]$, for any $a$, $2x = a$ has a solution $x$. 
But in ${\mathbb Z}^{\infty}$, $2x = c_1$ has no solution $x$. 
\end{rem}

\subsection
{Sierpinski Gasket case}
We calculate the K-group of the core of the \cst -algebra associated with
the self similar map which gives Sierpinski Gasket $S$.  

\begin{lemma} Let $S$ be the Sierpinski Gasket $S$. Then we have that 
\[
   {\rm K}_0({\rm C}(S)) = {\mathbb Z}, \quad    {\rm K}_1({\rm C}(S)) =
 {\mathbb Z}^{\infty},  
\]
\end{lemma}
\begin{proof} The K group of the Sierpinski Gasket $S$ is calculated by 
the inductive limit construction.  The inductive limit of the figures $S_n$
is obtained by cutting off 3 open discs from the unit disc successively is
 homeomorphic to $S$. Since  ${\rm K}_0({\rm C}(S_n))$ is isomorphic to
 ${\mathbb Z}$, we have that ${\rm K}_0({\rm C}(S))$ is also isomorphic to
 ${\mathbb Z}$.  Since the number of free generators of ${\rm K}_1$
 increases by  $3$, ${\rm K}_1({\rm C}(S))$ is isomorphic to ${\mathbb
 Z}^{\infty}$.  
\end{proof}

We calculate the K-groups of the finite cores of the \cst-algebra
associated with the Sierpinski Gasket using explicit calculation
of exp maps.  For the Sierpinski Gasket $S$ , K${}_1({\rm C}(S))$ contains many
elements in contrast to the tent map case.  

We denote the following algebras:
\begin{align*}
  D({\rm P})= & \Pi^{(n)}(\K(X^{\otimes n}))({\rm P})
  \oplus  \bigoplus_{p=0}^{n-1}{\rm C}({\rm P},{\rm T},p)
 \simeq M_{(1/2)(3^{n}+1)}(\comp) \oplus
 \bigoplus_{p=0}^{n-1}M_{3^{p}}(\comp),  \\
  D({\rm Q})= &\Pi^{(n)}(\K(X^{\otimes n}))({\rm Q}) \oplus
        \bigoplus_{p=0}^{n-1}H^{n-p+1}\simeq M_{(1/2)(3^{n}+1)}(\comp)
 \oplus \bigoplus_{p=0}^{n-1}M_{3^{p}}(\comp)), \\
  D({\rm R})= & \Pi^{(n)}(\K(X^{\otimes n}))({\rm R}) \oplus
        \bigoplus_{p=0}^{n-1}I^{n-p+1}\simeq M_{(1/2)(3^{n}+1)}(\comp)
 \oplus \bigoplus_{p=0}^{n-1}M_{3^{p}}(\comp),  \\
  J_n = & \{\,T \in {\rm C}(S,M_{N^n}(\comp))\,|\,T({\rm P})=O,\, T({\rm
 Q})=O,\,
 T({\rm R})=O\,\},   \\
  B = & ( \F^{(n)} = ) \{\,T \in {\rm C}(
 S,M_{N^n}(\comp))\,|\,T({\rm P}) \in D({\rm P}),\, T({\rm Q})
 \in D({\rm Q}),\, T({\rm R}) \in D({\rm R})\,\},   \\
  C = & D({\rm P}) \oplus D({\rm Q}) \oplus D({\rm R}).  
\end{align*}
We have the exact sequence:
\[
   \{0\} \to J_n \to B  \to C \to \{0\}.  
\]
Using this, it holds the following 6-term exact sequence:
\[
\begin{CD}
   {\rm K}_0(J_n) @>{}>> {\rm K}_0(\F^{(n)}) @>{}>> {\rm K}_0(C) \\
   @A{{\rm ind}}AA
    @.
     @VV{{\rm exp}}V \\
   {\rm K}_1(C) @<<{}< {\rm K}_1(\F^{(n)}) @<<{}< {\rm K}_1(J_n). 
\end{CD}
\]

We need to  consider $\tilde{J_n}$, which is the adjoining unit
\cst-algebra of
$J_n$,  when we calculate ${\rm K}_1(J_n)$.  We put $J=J_0$.  
Since $J_n = J_0 \otimes M_{3^n}(\comp)$, the K${}_1$-group of
$\tilde{J_n}$ is identical to the K${}_1$-group of $\tilde{J}$.  

To compute K${}_1(\tilde{J})$, we use the following 6 term
exact sequence of K-groups:
\[
\begin{CD}
   {\rm K}_0(J) @>{}>> {\rm K}_0({\rm C}(S)) @>{}>> {\rm K}_0(\comp^3) \\
   @A{{\rm ind}}AA
    @.
     @VV{{\rm exp}}V \\
      {\rm K}_1(\comp^3) @<<{}< {\rm K}_1({\rm C}(S)) @<<{}< {\rm
 K}_1(J).  
\end{CD}  
\]
Then it holds that 
\[
\begin{CD}
   {\rm K}_0(J) @>{}>>  {\mathbb Z} @>{}>> {\mathbb Z}^3 \\
   @A{{\rm ind}}AA
    @.
     @VV{{\rm exp}}V \\
       \{{\mathbf 0}\} @<<{}<  {\mathbb Z}^{\infty} @<<{}< {\rm
 K}_1(\tilde{J}).  
\end{CD}
\]
Since the identity $I$ in ${\rm C}(S)$ is mapped to $(1,1,1)$ in
$\comp^3$, the map from ${\rm K}_0({\rm C}(S))$ to ${\rm K}_0(\comp^3)$ is
injective.  It hollows that K${}_0(J)$ is equal to
$\{\mathbf 0\}$.  Then the following long exact sequence
\[
  \{\mathbf 0\} \to {\mathbb Z} \to {\mathbb Z}^3 \to {\rm
  K}_1(\tilde{J}) \to {\mathbb Z}^{\infty} \to \{{\mathbf 0}\}
\]
holds.  Then it holds that 
\[
  {\rm K}_1(\tilde{J}) \simeq \left({\mathbb Z}^3/{\mathbb Z}\right) \oplus {\mathbb
  Z}^{\infty} \simeq 
{\mathbb Z}^2  \oplus {\mathbb Z}^{\infty}.  
\]

We substitute known ${\rm K}$ group:
\[
\begin{CD}
   \{\mathbf 0\} @>{}>> {\rm K}_0(\F^{(n)}) @>{}>> {\mathbb Z}^{n+1}\oplus
 {\mathbb Z}^{n+1}\oplus {\mathbb Z}^{n+1}  \\
   @A{{\rm ind}}AA
    @.
     @VV{{\rm exp}}V \\
   \{\mathbf 0\}  @<<{}< {\rm K}_1(\F^{(n)}) @<<{}< {\rm K}_1(J_n).  
\end{CD}
\]
Since $C$ has a unit, we can write exp map explicitly using Proposition 
12.2.2 (ii) of Rordam {\it et al.} \cite{Ro}.  
Since the Sierpinski Gasket is connected, for projections $p_1$, $p_2$,
$p_3$, if their dimensions are not equal, there does not exist an
self-adjoint element $T \in \F^{(n)}$ such that $T({\rm P})=p_1$, $T({\rm Q})=p_2$ and
$T({\rm R})=p_3$, and if their dimensions are equal there exists an Hermit
element $T \in \F^{(n)}$ such that $T({\rm P})=p_1$, $T({\rm Q})=p_2$ and
$T({\rm R})=p_3$.  

By putting ${\mathbf m}^1=(m_1^1,\dots,m_n^1)$, ${\mathbf
m}^2=(m_1^2,\dots,m_n^2)$, ${\mathbf m}^3=(m_1^3,\dots,m_n^3)$, 
$|{\mathbf m}^i| = m_1^i + m_2^i + \dots + m_n^i$, it holds that 
\begin{align*}
    {\rm K}_0(\F^{(n)}) \simeq 
& \text{ The group generated by }
\{\,({\mathbf m}_1,{\mathbf m}_2,{\mathbf
    m}_3) \in {\mathbb Z}_+^{n+1}\oplus {\mathbb Z}_+^{n+1}\oplus {\mathbb
    Z}_+^{n+1}\,|\, |{\mathbf m}^1|=|{\mathbf m}^2|=|{\mathbf m}^3|\,\} \\
   = &\{\,((m_1,\dots,m_{n+1}),(r_1,\dots,r_{n+1}),(s_1,\dots,s_{n+1}))
   \in {\mathbb Z}^{n+1} \oplus {\mathbb Z}^{n+1} \oplus {\mathbb
   Z}^{n+1} \\
  & \sum_{i=1}^{n+1}m_i=\sum_{i=1}^{n+1}r_i=\sum_{i=1}^{n+1}s_i\,\}.  
\end{align*}

We calculate the embedding of K groups of the finite cores derived from
the embedding from $\F^{(n)}$ to $\F^{(n+1)}$.  The generators of ${\rm
K}_0(\F^{(n)})$ are given by paths of projections in $\F^{(n)}$, and
they are also paths of
projections in $\F^{(n+1)}$.  The embedding is calculated using the
matrix representation.  For the calculation, we present the generators
of ${\rm K}_0$-groups of $D({\rm P})$, $D({\rm Q})$ and $D({\rm R})$.
Let $p^{\rm P}_i$ be a
minimal projection of ${\rm C}({\rm P},T,i+1)$, 
$p^{\rm Q}_i$ be a minimal projection of ${\rm H}^{n-i}$ and
$\Pi^{(n)}(\K(X^{\otimes
n}))({\rm Q})$, $p^{\rm R}_i$ be a minimal projection of ${\rm
I}^{n-i}$  for $1 \le i \le n$.  Let $p^{\rm P}_{n+1}$ be a minimal
projection in
$\Pi^{(n)}(\K(X^{\otimes n}))({\rm P})$, $p^{\rm Q}_{n+1}$ be a a
minimal projection
in $\Pi^{(n)}(\K(X^{\otimes n}))({\rm Q})$ and $p^{{\rm R}}_{n+1}$ be a
minimal projection in $\Pi^{(n)}(\K(X^{\otimes n}))({\rm R})$.  
Then it holds
\begin{align*}
 {\rm K}_0(D({\rm P})) = &\{\,m^1_1[p^{\rm P}_1] + \dots +
 m^1_{n+1}[p^{\rm P}_{n+1}]\,|\,m^1_i \in {\mathbb Z}\,\}
 \simeq \{\,(m^1_1,\dots,m^1_{n+1})\,|\,m^1_i \in {\mathbb Z}\,\} \\
 {\rm K}_0(D({\rm Q})) = &\{\,m^2_1[p^{\rm Q}_1] + \dots +
 m^2_{n+1}[p^{\rm Q}_{n+1}]\,|\,m^2_i \in {\mathbb Z}\,\}
 \simeq \{\,(m^2_1,\dots,m^2_{n+1})\,|\,m^2_i \in {\mathbb Z}\,\} \\
 {\rm K}_0(D({\rm R})) = &\{\,m^3_1[p^{\rm R}_1] + \dots +
 m^3_{n+1}[p^{\rm R}_{n+1}]\,|\,m^3_i \in {\mathbb Z}\,\}
 \simeq \{\,(m^3_1,\dots,m^1_{n+1})\,|\,m^3_i \in {\mathbb Z}\,\}.   
\end{align*}

We denote by $\Psi^{n,n+1}$ the map of embedding from ${\rm
K}_0(\F^{(n)})$ to ${\rm K}_0(\F^{(n+1)})$.  
\begin{align*}
 &
 \Psi^{n,n+1}((0,\dots,1_{i^{\rm P}},\dots,0),(0,\dots,1_{i^{\rm Q}},\dots,0),
(0,\dots,1_{i^{\rm R}},\dots,0)) \\
 = &((0,\dots,1_{i^{\rm P}+1},\dots,0),(0,\dots,1_{{i^{\rm R}}+1},\dots,0),
 (0,\dots,1_{{i^{\rm Q}}+1},\dots,0)) \\
 & +  ((1_1,\dots,0,\dots,1_{n+2}),(1_1,\dots,0,\dots,1_{n+2}),
 (1_1,\dots,0,\dots,1_{n+2})).  
\end{align*}
Then K${}_0(\F^{(\infty)})$ is the inductive limit group of the above 
inclusion maps $\{\Psi^{(n,n+1)}\}_{n=0,1,2,\dots}$.  

We take a basis of 
${\rm K}_0(\F^{(\infty)}) \simeq {\mathbb Z}^{3n+1}$ as follows:
\begin{align*}
 a^{n}_i = & ((0,\dots,0,1_i,0,\dots,0),(0,\dots,0,1_{n+1}),(0,0,\dots,1_{n+1})) 
   \quad i=1,\dots,n+1, \\
 b^{n}_i = & ((1_1,0,\dots,0),(0,\dots,0,1_{i},0,\dots,0),(0,\dots,0,1_{n+1}))
 \quad i=1,\dots,n, \\
 c^n_i = & ((1_1,0,\dots,0),(0,\dots,0,1_{n+1}),(0,\dots,0,1_i,0,\dots,0)) 
 \quad i=1,\dots,n.   
\end{align*}
We do the following preliminary calculation:
\begin{align*}
  & ((1_1,0,\dots,0,1_{n+2}),
 (1_1,0,\dots,0,1_{n+2}),(1_1,0,\dots,0,1_{n+2})) \\
= & ((0,\dots,0,1_{n+2}), (0,\dots,0,1_{n+2}),(0,\dots,0,1_{n+2})) 
  +  ((1_1,0,\dots,0), (1_1,0,\dots,0),(1_1,0,\dots,0)) \\
 = & a^{n+1}_{n+2} 
  + ((1_1,\dots,0),(1_1,\dots,0),(0,\dots,1_{n+1})) 
  + ((0,\dots,0),(0,\dots,0),(1_{1},\dots,-1_{n+1})) \\
 = & a^{n+1}_{n+2}  + b^{n+1}_1 + (c^{n+1}_1 - a^{n+1}_1) 
 =  - a^{n+1}_1 + a^{n+1}_{n+2}  + b^{n+1}_1 + c^{n+1}_1.  
\end{align*}
We calculate the action of $\Phi^{(n,n+1)}$ to each basis:
\begin{align*}
 \Phi^{(n,n+1)}(a^n_i)
   = &
 ((0,\dots,0,1_{i+1},0,\dots,0),(0,\dots,0,1_{n+1}),(0,\dots,0,1_{n+1})) \\
  &  + ((1_1,0,\dots,0,1_{n+2}),
 (1_1,0,\dots,0,1_{n+2}),(1_1,0,\dots,0,1_{n+2})) \\
  = & a^{n+1}_{i+1} + (- a^{n+1}_1 + a^{n+1}_{n+2}  + b^{n+1}_1 +
 c^{n+1}_1 ) \\
 = & -a^{n+1}_1 + a^{n+1}_{i+1} + a^{n+1}_{n+2}  + b^{n+1}_1 +
 c^{n+1}_1.  \\
  \Phi^{(n,n+1)}(b^n_i)
 = &
 ((0,1_1,0,\dots,0),(0,\dots,0,1_{i+1},0,\dots,0),(0,\dots,0,i_{n+2})) \\
   &  + ((1_1,0,\dots,0,1_{n+2}),
 (1_1,0,\dots,0,1_{n+2}),(1_1,0,\dots,0,1_{n+2})) \\
 = & ((0,1_1,0,\dots,0),(0,\dots,0,1_{n+2}),(0,\dots,0,i_{n+2}))  \\
   &  + ((0,\dots,0),(0,\dots,0,1_{i+1},0,\dots,-1_{n+2}),(0,\dots,0)) \\
   &  + ((1_1,0,\dots,0,1_{n+2}),
 (1_1,0,\dots,0,1_{n+2}),(1_1,0,\dots,0,1_{n+2})) \\
 =&  a^{n+1}_2 +(b^{n+1}_{i+1}-a^{n+1}_1)+ (- a^{n+1}_1 + a^{n+1}_{n+2}  +
 b^{n+1}_1 + c^{n+1}_1 )\\
 = & -2a^{n+1}_1 + a^{n+1}_2 + a_{n+2}^{n+1} + b^{n+1}_1 + b^{n+1}_{i+1}
 + c^{n+1}_1.  
\\
 \Phi^{(n,n+1)}(c^n_i) 
 = &
 ((0,1_1,0,\dots,0),(0,\dots,0,1_{n+2}),(0,\dots,0,1_{i+1},0,\dots,0)) \\
 &   + ((1_1,0,\dots,0,1_{n+2}),
 (1_1,0,\dots,0,1_{n+2}),(1_1,0,\dots,0,1_{n+2})) \\
 =&  ((0,1_1,0,\dots,0),(0,\dots,0,1_{n+2}),(0,\dots,0,1_{n+1})) \\
  &  + ((0,\dots,0),(0,\dots,0),(0,\dots,0,1_{i+1},0,\dots,-1_{n+2})) \\
 &   + ((1_1,0,\dots,0,1_{n+2}),
 (1_1,0,\dots,0,1_{n+2}),(1_1,0,\dots,0,1_{n+2})) \\
 = & a^{n+1}_2 + (c^{n+1}_{i+1} -a^{n+1}_1)+ 
      (- a^{n+1}_1 + a^{n+1}_{n+2}  + b^{n+1}_1 + c^{n+1}_1 )\\
 = & -2a^{n+1}_1 + a^{n+1}_2 + a_{n+2}^{n+1}+ b^{n+1}_1 + c^{n+1}_1 +
 c^{n+1}_{i+1}.  
\end{align*}
For the map $\Psi^{(n,n+1)}$, we can define a matrix $B^{(n,n+1)}$ using
the above bases.  

For $n=3$, we have
\[
 B^{(n,n+1)} = \begin{pmatrix} 
           -1 & -1 & -1 & -1 &-2 & -2& -2& -2 & -2 & -2 \\
            1 & 0  &  0 &  0 & 1 & 1 & 1 & 1 & 1 & 1 \\
            0 & 1  &  0 &  0 & 0 & 0 &  0& 0 & 0 & 0\\
            0 & 0  &  1	&  0 & 0 & 0 &  0& 0 & 0 & 0\\
            1 & 1  &  1	&  2 & 1 & 1 & 1 & 1 & 1 & 1 \\
            1 &	1  &  1 &  1 & 1 & 1 & 1 & 1 & 1 & 1\\
            0 & 0  &  0 &  0 & 1 & 0 & 0 & 0 & 0 & 0\\
            0 & 0  &  0 &  0 & 0 & 1 & 0 & 0 & 0 & 0\\
            0 & 0  &  0	&  0 & 0 & 0 & 1 & 0 & 0 & 0\\
            1 & 1  &  1 &  1 & 1 & 1 & 1 & 1 & 1 & 1\\
            0 & 0  &  0 &  0 & 0 & 0 & 0 & 1 & 0 & 0 \\
            0 & 0  &  0 &  0 & 0 & 0 & 0 & 0 & 1 & 0 \\
            0 & 0  &  0	&  0 & 0 & 0 & 0 & 0 & 0 & 1 \\
	       \end{pmatrix}.  
\]
To state the form for general $n$, we prepare some notations. 
For $a \in {\mathbb Z}$, $a_m$ denotes the row vector whose entries are
all $a$ and $a^m$ denotes the column vector whose entries are all $a$.   
$E_m$ denotes the identity matrix of dimension $m$, $O_m$ denotes
the zero matrix of dimension $m$ and $F_m$ denotes the $m$ dimensional
matrix such that $(F_m)_{1,j}=1$ and $(F_m)_{i,j}=0$ for $2 \le i \le m$ and
$1 \le j \le m$.  For general $n$, we have 
\begin{equation} \label{eq:inclusion}
 B^{(n,n+1)}
   = \begin{pmatrix}
    (-1)_n &  -1 & (-2)_n &  (-2)_n    \\
     E_n   &  0^n & F_n   & F_n   \\
     1_n   &   2  & 1_n   & 1_n   \\   
     1_n   &   1  & 1_n   & 1_n   \\
     O_n   &   0^n & E_n  &  O_n \\
     1_n   &  1    & 1_n  &  1_n \\ 
     O_n   &  0^n  & O_n  &  E_n 
     \end{pmatrix} 
\end{equation}

At last, we consider the ${\rm K}_1$ group of $\F^{(n)}$.  
The image of the exp map is ${\mathbb Z}^2$.  
We write a generators of the range of exp map explicitly.  
Let $\tilde{S}$ be the set obtained by identifying P, Q, R from the
Sierpinski gasket S.  A continuous map $h_{{\rm P}{\rm Q}}$ from $\tilde{S}$ to $\mathbb T$
whose winding number is equal $1$ around the circle from P to R, and
$0$ around the circle from P to S and R to S.  Then $[h_{{\rm P}{\rm Q}}]-[h_{{\rm P}{\rm R}}]$
is a generator of the image of exponential map.  Another generator is
constructed analogously.  
Then ${\rm K}_1(\F^{(n)})$ is isomorphic to ${\mathbb Z}^{\infty}$ and
the inclusion map from ${\rm K}_1(\F^{(n)})$ to ${\rm K}_1(\F^{(n+1)})$
is identity map by the canonical identification.  Then 
K${}_1(\F^{\infty})$ is isomorphic to K${}_1(S) \simeq {\mathbb
Z}^{\infty}$.  

\begin{prop} \label{prop:K-group-Sierpinski} Let $\F^{(\infty)}$ be the core of the \cst-algebra
 associated with the Sierpinski Gasket. Then we have that 
\begin{align*}
 {\rm K}_0(\F^{(\infty)}) \simeq & \lim_{n \to \infty}
          \left({\mathbb Z}^{3n+1} \overset{B^{(n,n+1)}}{\to} {\mathbb
 Z}^{3n+4} \right)\\
   {\rm K}_1(\F^{(\infty)}) \simeq & {\mathbb Z}^{\infty}.  
\end{align*}
\end{prop}

\begin{lemma} \label{lem:injective}
Each map $\Psi^{(n,n+1)}$ is injective.  
 \end{lemma}
\begin{proof}
By the expression \eqref{eq:inclusion}, we can show that 
the rank of the matrix $B^{(n,n+1)}$ is $3n+1$.  We show it for the case
 $n=3$.  By row basic deformation, the matrix $B^{(3,4)}$ is
 transformd to the following: 
\[
 \begin{pmatrix} 
            0 & 0  &  0 &  0 & 0 & 0 & 0 & 0 &  0 & 0 \\
            1 & 0  &  0 &  0 & 0 & 0 & 0 & 0 & 0 & 0  \\
            0 & 1  &  0 &  0 & 0 & 0 & 0& 0 & 0 & 0\\
            0 & 0  &  1	&  0 & 0 & 0 &  0& 0 & 0 & 0\\
            0 & 0  &  0	&  0 & 0 & 0 &  0& 0 & 0 & 0\\
            0 & 0  &  0	&  1 & 0 & 0 &  0& 0 & 0 & 0\\
            0 & 0  &  0 &  0 & 1 & 0 & 0 & 0 & 0 & 0\\
            0 & 0  &  0 &  0 & 0 & 1 & 0 & 0 & 0 & 0\\
            0 & 0  &  0	&  0 & 0 & 0 & 1 & 0 & 0 & 0\\
            0 & 0  &  0	&  0 & 0 & 0 &  0& 0 & 0 & 0\\
            0 & 0  &  0 &  0 & 0 & 0 & 0 & 1 & 0 & 0 \\
            0 & 0  &  0 &  0 & 0 & 0 & 0 & 0 & 1 & 0 \\
            0 & 0  &  0	&  0 & 0 & 0 & 0 & 0 & 0 & 1 \\
	       \end{pmatrix}.  
\]
The general $n$ case is similar.  
Then $B^{(n,n+1)}$ is injective as linear map from ${\mathbb R}^{3n+1}$ to
 ${\mathbb R}^{3n+4}$.  
\end{proof}


\section
{Dimension groups}
The dimension groups for topological Markov shifts were 
introduced and studied by 
Krieger in \cite{Kr1} and \cite{Kr2} motivated by the K-groups 
for the AF-subalgebras of the associated Cuntz-Krieger algebras \cite{CK}.
The dimension group is a complete invariant for topological Markov shifts 
up to shift equivalence. K. Matsumoto \cite{Ma2} 
studied the dimension groups for the \cst-algebras 
$\O_{\Lambda}$ associated with subshifts $\Lambda$. He showed that 
the dual action $\hat{\alpha}$ of the gauge action $\alpha$ 
on the \cst-algebra  $\O_{\Lambda}$ 
induces an automorphism $\beta$ on the  K-group 
$K_0(F_{\Lambda})$ of the fixed point 
algebra $F_{\Lambda}$ through  Morita 
equivalence between the fixed point algebra
 and the crossed product algebra by the gauge action. A 
system $(K_0(F_{\Lambda}), K_0(F_{\Lambda})_+, \beta)$
 is called the dimension group and 
an important invariant for symbolic dynamical systems. 

In the case of self-similar maps without branched points, 
the same method works. But 
in the case of self-similar maps with branched points, 
it does not work, since the fixed point algebra
 and the crossed product algebra by the gauge action 
are not Morita equivalent in general. 
However we shall show that it is possible to
introduce a certain canonical {\it endomorphism} 
(similar to the automorphism above) 
on the K-group of the core using isometries in the one-spectral 
subspace directly.  

Let $G$ be a compact abelian group. Consider a \cst-dynamical system 
$(A, G, \alpha)$, that is, $A$ is a \cst -algebra and 
$\alpha$ is a homomorphism from $G$ to the automorphism group 
${\rm Aut}(A)$  of $A$. We assume that $A$ is unital. 
For $a$, $b \in A$ and $g \in G$, 
put $f_{a,b}(g)= a \alpha_g(b)$. Then $f_{a,b} \in {\rm C}(G,A)$ is an element 
in the crossed product $A \rtimes_{\alpha} G$.  We sometimes 
write it as $\int_{G} f_{a,b}(z)\lambda _z dz$. 
Define a constant function $e \in {\rm C}(G,A)$ by  $e(g) = 1$ for $g \in G$. 
The $e$ is a projection and we sometimes write it as  
\[
e =\int_{\mathbb G} 1 \lambda _z dz \in A \rtimes_{\alpha}{G} \ \ , 
\]
We can also identify $f_{a,b}$ with "rank one operator" $aeb$ on the 
Hilbert $A^{\alpha}$-module $A$. 

There exists an isomorphism of the fixed point algebra 
$A^{\alpha}$ onto the hereditary subalgebra $e(A \rtimes_{\alpha}{G})e$
such that $a \mapsto f_{a,1} = ae = ea$ for $a \in A^{\alpha}$ \cite{Ros}. 

\begin{defn} (Rieffel, see Phillips \cite{Ph}) 
Let $G$ be a compact abelian group and $(A, G, \alpha)$ a \cst-dynamical
 system.  Then the action 
$\alpha$ is called saturated if the linear span of 
$\{f_{a,b} \ | \ a,b \in A\}$ is dense in $A \rtimes_{\alpha}{G}$. 
The closed linear span of $\{f_{a,b} \ | \ a,b \in A\}$ is the 
ideal generated by $e$. Therefore 
 the action 
$\alpha$ is saturated if $e$ is a full projection in 
$A \rtimes_{\alpha}{G}$. 
\end{defn}

We can associate  left $A^{G}$- right $A \rtimes_{\alpha}G$-module
$Y$ as follows: We put $Y=A$ as a linear space.  For $a \in A^{G}$, $f\in {\rm
C}(G,A)$ and $x \in A$, we define module actions by 
\[
  a\cdot x = ax, \qquad x \cdot f = \int_G \alpha_{g^{-1}}(xf(g))\,{\rm
  d}g.   
\]
For $x$, $y \in A$, we define two inner products by 
\[
 {}_{A^{\alpha}} (x|y) = \int_G\alpha_g(xy^*)\,{\rm d}g, \quad 
 (x|y)_{A \rtimes_{\alpha}G} = ``g \mapsto x^*\alpha_g(y) ``.  
\]
Then the action 
$\alpha$ is saturated if and only if $A^G$-$A\rtimes_{\alpha}G$ bimodule $Y$
is a Morita equivalence module.  

For example, the gauge action on a Cuntz-Krieger algebra is saturated. 
Matsumoto \cite{Ma2} showed that the gauge action on the \cst-algebra 
$\O_{\Lambda}$ associated with a subshifts $\Lambda$ is saturated. 
J. Jeong \cite{Je} showed that the gauge action on the graph \cst-algebra 
associated with a locally finite directed graph with no sources or sinks 
is saturated. 

The following theorem is a generalization to the case of 
Cuntz-Pimsner algebras.

\begin{prop} \label{prop:saturated}
 Let $A$ be a unital \cst-algebra. 
Let $X$ be an $A$-$A$ correspondence.  We assume that $X$ is full,
the left module action of $A$ on $X$ is injective and there exits
$x_0 \in X$ such that $(x_0|x_0)_A = I$. 
If $X$ has a finite  basis  as a Hilbert right $A$-module, then the gauge
action $\alpha$ of ${\mathbb T}$ on $\O_{X}$ is saturated.  
\end{prop}
\begin{proof}
 Let $B=\O_X$ and $G = {\mathbb T}$. 
We shall consider a hat map of 
$B=\O_X$ into  $\O_{X} \rtimes_{\alpha}{\mathbb T}$:
For $b \in B$,  a constant function 
$\hat{b} = (z \in  {\mathbb T} \mapsto b)$ 
gives the desired element in $\O_{X} \rtimes_{\alpha}{\mathbb T}$.  
For $a$, $b \in B$, put $f_{a,b}(z)= a \alpha_z(b)$. Put $e = \hat{1}$.
Then $e$ is a projection.  
We identify 
\[
e =\int_{\mathbb T} 1 \lambda _z dz \in \O_{X} \rtimes_{\alpha}{\mathbb 
 T}, \quad \hat{b} = be \text{ and } 
f_{a,b} = \int_{\mathbb T} f_{a,b}(z)\lambda _z dz
\]
as elements in 
the crossed product $\O_{X} \rtimes_{\alpha}{\mathbb T}$.
We can also identify $f_{a,b}$ with "rank one operator" $aeb$ on the 
Hilbert $\O_X^{\alpha}$-module $\O_{X}$.  

Since 
$\O_{X} \rtimes_{\alpha}{\mathbb T}$ is the closure of the span 
by 
\[
\{\, z \mapsto f(z)T \,| \, f \in C({\mathbb T}), T \in \O_{X} \,\},   
\]
it is enough to show that for each $n \in {\mathbb Z}$, $x \in
 X^{\otimes m}$, $y \in X^{\otimes k}$ and $c \in A$, 
\[
 ``z \in {\mathbb T} \mapsto z^n S_xS_y^* `` \in {\rm Span }\{\,f_{a,b}\,|\, a,b
 \in \O_X \,\},   
\]
and 
\[
 ``z \in {\mathbb T} \mapsto z^n c  `` \in {\rm Span }\{\,f_{a,b}\,|\, a,b
 \in \O_X\,\}.  
\]

Firstly consider the case that  $n > (|x|-|y|)$.  Put $l = n-(|x|-|y|) > 0$.  
Put $v = x_0 \otimes \cdots \otimes x_0 \in X^{\otimes l}$.  Then 
$S_v^*S_v=1$, and it holds 
\begin{align*}
  z^nS_xS_y^* =  & S_v^*z^n S_vS_xS_y^*  \\
              = & S_v^* \alpha_z(S_vS_xS_y^*) \\
              = & f_{a,b}(z), 
\end{align*}
where $a=S_{v}^*\in B$, $b = S_vS_xS_y^* \in B$.  Secondly we 
consider the case that $ n-(|x|-|y|)=0$.  
Then it holds that
\[
    z^n S_x S_y^*  = 1 \alpha_z(S_xS_y^*) = f_{a,b}(z), 
\]
where $a=1 \in B$ and $b=S_xS_y^* \in B$. Thirdly we 
consider the case that  $n-(|x|-|y|)<0$, and put $l=(|x|-|y|)-n > 0$.  
Let $\{\,u_1,\dots,u_p\,\} \in X$  be a finite
basis. Then 
$$
\{\,u_{i_1}
\otimes \cdots \otimes u_{i_n}\,|\,(i_1,\dots,i_n) \in \{1,\dots,p\}^n\,\}
$$ is
a finite bases of $X^{\otimes n}$.  Hence  we have that 
\begin{align*}
    z^n S_x S_y^* = & z^n  \sum_{(i_1,\dots,i_n) \in \{1,\dots,p\}^n}
    S_{u_{i_1}}\cdots S_{u_{i_n}} 
      S_{u_{i_n}}^*\cdots S_{u_{i_1}}^* S_xS_y^* \\
      = & \sum_{(i_1,\dots,i_n) \in \{1,\dots,p\}^n}
    S_{u_{i_1}}\cdots S_{u_{i_n}}S_{u_{i_n}}^*\cdots S_{i_{i_{l-1}}}^*
      \alpha_z(S_{u_{i_l}}^*\cdots S_{u_{i_1}}^* S_xS_y^*) \\
      = & \sum_{(i_1,\dots,i_n) \in
 \{1,\dots,p\}^n}f_{a_{i_1,\dots,i_n},b_{i_1,\dots,i_n}}(z), 
\end{align*}
where $a_{i_1,\dots,i_n}=S_{u_{i_1}}\cdots S_{u_{i_n}}S_{u_{i_n}}^*\cdots
 S_{u_{i_{l+1}}}^*$ and 
$b_{i_1,\dots,i_n}=S_{u_{i_{l}}}^*\cdots S_{u_{i_1}}^* S_xS_y^*$.  
\end{proof}

The proof in Proposition \ref{prop:saturated} does not work if $X$ does 
not have a finite basis. 

\begin{thm}  Let $A$ be a unital $C^*$-algebra. 
Let $X$ be an $A$-$A$ correspondence.  We assume that $X$ is full,
 left module action of $A$ on $X$ is injective and there exits $x_0 \in
 X$ such that $(x_0|x_0)_A = I$. 
Suppose  that $X$ has a finite  basis  as a Hilbert right $A$-module. 
For the gauge action $(\O_X,{\mathbb T},\alpha)$ and its dual action 
$(\O_X \rtimes {\mathbb T}, \hat{\mathbb T},\hat{\alpha})$, consider the 
following diagram:
\[
\begin{CD}
   {\rm K}_0(\O_X \rtimes_{\alpha}{\mathbb T}) @>{\hat{\alpha}_*}>>
{\rm K}_0(\O_X \rtimes_{\alpha}{\mathbb T}) \\
   @A{\varphi_*}AA
     @VV{\varphi_*^{-1}}V \\
   {\rm K}_0(\O_X^{\alpha}) @>{\beta}>> {\rm K}_0(\O_X^{\alpha})
\end{CD}
\]
Where $\varphi (b) = \hat{b}= be$ for $b \in \O_X^{\alpha}$. Then 
$e$ is a full projection in $\O_X \rtimes_{\alpha}{\mathbb T}$ 
and $\varphi_*$ gives an isomorphism of ${\rm K}_0(\O_X^{\alpha})$ onto 
${\rm K}_0(\O_X \rtimes_{\alpha}{\mathbb T})$. 
Moreover if we put $\beta := {\varphi_*^{-1}}{\hat{\alpha}_*}{\varphi_*}$, 
then the above diagram is commutative. For 
 each projection $P \in \O_X^{\alpha}$, any partial isometry $S$ in the 
one-spectral subspace $\O_X^{(1)}$ of $\O_X$ (i.e., 
$\alpha_t(S) = e^{it}S$ ) with $P \le S^*S$,  
we have $\beta([P]) = [SPS^*]$  in the $K_0$-group 
${\rm K}_0(\O_X^{\alpha})$ of the core. 
. 
\end{thm}
\begin{proof}
By Proposition \ref{prop:saturated}, 
$e$ is a full projection in $\O_X \rtimes_{\alpha}{\mathbb T}$, 
and $e$ commutes with the $O_X^{\alpha}$  and 
$O_X^{\alpha}e = e(\O_X \rtimes_{\alpha}{\mathbb T})e $. 
Therefore  $\varphi_*$ gives an isomorphism of ${\rm K}_0(\O_X^{\alpha})$ onto 
${\rm K}_0(\O_X \rtimes_{\alpha}{\mathbb T})$.
Since we put $\beta = \varphi_*^{-1}{\hat{\alpha}_*} \varphi_*$, 
$\beta$ gives an automorphism on the K-group ${\rm K}_0(\O_X^{\alpha})$.

Let $P \in \O_X^{\alpha}$ be a projection,  
$S$ a partial isometry in the 
one-spectral subspace of $\O_X$ such that $P \le S^*S$. 
Then $SPS^*$ is also a projection in $\O_X^{\alpha}$. Since 
$eP = Pe$ and $eSPS^* = SPS^*e$, $Pe$ is a projection and 
$SPe$ is a partial isometry.  In fact 
\[
(SPe)^*(SPe) = ePS^*SPe = eP, \ \ (SPe)(SPe)^* = SPePS^* = SPeS^*.    
\]
Therefore two projections $Pe$ and $SPeS^*$ are von Neumann 
equivalent in $\O_X \rtimes_{\alpha}{\mathbb T}$. 
It is enough to show that 
\[
\hat{\alpha}_*{\varphi_*}([P])=[ \hat{\alpha} (Pe)]= [SPS^*e].  
\]
We see that  $\hat{\alpha}(Se) = eS$. In fact, 
\[
eS = (\int_{\mathbb T} 1 \lambda _z dz)S 
= \int_{\mathbb T} \alpha_z(S) \lambda _z dz 
= \int_{\mathbb T} zS \lambda _z dz = \hat{\alpha}(Se).  
\]
Hence we have $\hat{\alpha}(eS^*)= S^*e$. Therefore
\[
 \hat{\alpha} (Pe)] = [ \hat{\alpha} (SPeS^*)]
= [\hat{\alpha} (SP) \hat{\alpha}(eS^*)] = [SPS^*e]
\]
\end{proof}

We note that $[SPS^*]$ does not depend on the choice of such $S$.  

The following is due to Rieffel as in Phillips \cite{Ph} 7.1.15.  

\begin{thm} \rm (Rieffel) 
Let  $(A, G, \alpha)$ be a \cst-dynamical system with a compact abelian
group $G$. For $\tau \in \hat{G}$, define the spectral subspace  $A_\tau$ by 
\[
  A_{\tau} = \{\,a \in A\,|\,\alpha_g(a)=\tau(g)a\,\}.  
\]
Then $\alpha$ is saturated if and only if
 $\overline{A_{\tau}^*A_{\tau}}=A^{\alpha}$ for each $\tau \in
 \hat{G}$.  
\end{thm}

\begin{prop} \rm Let $A$ be a unital \cst-algebra. 
Let $X$ be an $A$-$A$ correspondence.  We assume that $X$ is full,
 left module action of $A$ on $X$ is injective and there exits $x_0 \in
 X$ such that $(x_0|x_0)_A = I$. 
Suppose that $X$ does not have a finite  basis  as a Hilbert right
$A$-module.  Then the gauge action $({\mathbb{T}},\O_X,\alpha)$
 is not saturated.  
\end{prop}
\begin{proof}
The spectral subspace $\O_X^{(-1)}$ is defined by
\[
  \O_X^{(-1)} = \{\,T \in \O_X\,|\, \alpha_t(T) = e^{-it}T\,\}.  
\]
On the contrary, suppose that the gauge action $(\O_X, {\mathbb{T}},\alpha)$
 is saturated.  
Then $\overline{(\O_X^{(-1)})^* \O_X^{(-1)}}=\O_X^{\mathbb
 T}$. 
Since $S_{x_0}^*S_{x_0}=I$, 
for any $T \in \O_X^{(-1)}$ and $\varepsilon > 0$ there exist 
a family $\{B_k\}_{k=1,\dots,p} \subset {\F}^{(n)}$ for sufficiently
 large $n$ such that 
\[
   \| T -\sum_{k=1}^p S_{x_0}^* B_k\| < \varepsilon/2
\] 
Since $I$ is in $\O_X^{\mathbb T}$, for any $\varepsilon > 0$, 
there exist families 
$\{B_k^i\}_{i=1,\dots,s}$, $\{C_{k'}^i\}_{i'=1,\dots,s} \subset
 {\F}^{(n)}$ for sufficiently large $n$  such that 
\begin{align*}
 & \| I - \sum_{i=1}^s \left( \sum_{k=1}^{p}S_{x_0}^*B_k^i\right)^* 
   \left( \sum_{k'=1}^{q} S_{x_0}^* C_{k'}^i\right)\| < \varepsilon \\
 & \| I - \sum_{i=1}^s \sum_{k,k'} {B_k^i}^*S_{x_0} S_{x_0}^* C_{k'}^i\|
 < \varepsilon.  
\end{align*}
Therefore $I$ can be approximated by elements of $\K(X^{\otimes
 n})$. This contradicts that $X_A$ does not have a finite basis.  
Hence the gauge action is not saturated.  
\end{proof}

The fixed point algebra and the crossed product algebra by the gauge action 
are not Morita equivalent in general. However it is possible to
introduce an endomorphism (similar to the automorphism as above) on the
K-group of the core using isometries in the one-spectral subspace
directly.  

\begin{prop} \rm Let $A$ be a unital \cst-algebra. 
Let $X$ be an $A$-$A$ correspondence.  We assume that $X$ is full,
 left module action of $A$ on $X$ is injective and there exits $x_0 \in
 X$ such that $(x_0|x_0)_A = I$. 
Then there exists an endomorphism $\beta$ on the ${\rm K}_0$-group 
$({\rm K}_0(\O_X^{\alpha}), {\rm K}_0(\O_X^{\alpha})_+)$ of the core satisfying the following: 
For any projection $P \in \O_X^{\alpha}$, any isometry $S$ in the 
one-spectral subspace $\O_X^{(1)}$ of $\O_X$ 
we have $\beta([P]) = [SPS^*]$  in the ${\rm K}_0$-group 
${\rm K}_0(\O_X^{\alpha})$ of the core. The endomorphism $\beta$ 
does not depend on the choice of such isometry $S$.
\end{prop}
\begin{proof}
For any isometry $S$ in the one-spectral subspace $\O_X^{(1)}$ of $\O_X$, 
$\delta_S(T) = STS^*$ for $T \in \O_X^{\alpha}$ gives an *-endomorphism 
$\delta_S$ on the fixed point algebra $\O_X^{\alpha}$. Therefore 
$\delta_S$ induces an endomorphism $\beta$  on the $K_0$-group 
$({\rm K}_0(\O_X^{\alpha}), {\rm K}_0(\O_X^{\alpha})_+)$ of the core such that 
$\beta([P]) = [SPS^*]$ for any projection $P \in \O_X^{\alpha}$.  
Take another isometry $W$ in the 
one-spectral subspace $\O_X^{(1)}$ of $\O_X$. 
Then 
\[
(WPS^*)^*(WPS^*) = SPW^*WPS^* = SPS^*, \ \ \ (WPS^*)(WPS^*)^* = WPW^*. 
\]
Therefore $SPS^*$ and $WPW^*$ are von Neumann equivalent and 
$[SPS^*] = [WPW^*]$.  
\end{proof}

\begin{defn}  \rm In the above situation, we call the system 
$({\rm K}_0(\O_X^{\alpha}), {\rm K}_0(\O_X^{\alpha})_+, \beta)$ 
the dimension group for the bimodule $X$.  Moreover, 
let $\gamma=(\gamma_0,\dots,\gamma_{N-1})$ be a self-similar map 
on a compact set $K$ and $X = X_{\gamma}$ is the associated 
bimodule. Then we also call the system 
$({\rm K}_0(\O_X^{\alpha}), {\rm K}_0(\O_X^{\alpha})_+, \beta)$ 
the dimension group for the original self-similar map 
$\gamma=(\gamma_0,\dots,\gamma_{N-1})$. 

\end{defn}

We shall calculate this endomorphism on the K-group of the core for some 
cases that the \cst-correspondence which does not have a 
finite basis, and show that the endomorphism is not surjective in general.  

\begin{thm}
Let $(K,\gamma)$ be the self-similar map given by the tent map.  
Then the dimension group is isomorphic to the countably generated 
free abelian group 
${\mathbb Z}^{\infty} = \coprod _{n=0}^{\infty}{\mathbb Z} 
\cong {\mathbb Z}[t]$ as an abstract group and 
the canonical endomorphism 
$\beta$ can be identified with the unilateral shift 
on it i.e. the 
multiplication map by $t$ .The
canonical endomorphism $\beta$ of the dimension group is not surjective.
\end{thm}
\begin{proof}
Let $S_1 = (1/\sqrt{\,2}){\bf 1}_{\C_{\gamma}}$. Then 
the endomorphism $\beta$ is given by
$\beta([P]) = [S_1PS_1^*]$. For each projection $P$ in $\F^{(n)}$,
$\beta([P])$ is expressed as $[P \otimes q]$ where 
\[
          q = \frac{1}{2} \begin{pmatrix}
               1 & 1   \\ 1 & 1
	      \end{pmatrix}.  
\]

We identify  the  dimension group with 
the countably generated 
free abelian group ${\mathbb Z}^{\infty}$ 
generated by $c_n = \varphi([p^n])$ in  Theorem \ref{thm:dimension-group}. 
Since 
$$
\beta([p^n]) = [p^n \otimes q] = [p^{n+1}]
$$
we see that 
the endomorphism 
$\beta$ on the dimension group can be identified with the unilateral shift 
on it.
\end{proof}

\end{document}